\documentclass[final,leqno,onefignum,onetabnum]{siamart190516}

\usepackage{amsfonts}
\usepackage{amsmath}
\usepackage{bm}
\usepackage{graphicx}
\usepackage{epstopdf}
\usepackage{caption}
\usepackage{subcaption}
\usepackage{algorithm}
\usepackage{algpseudocode}
\usepackage{setspace}
\ifpdf
  \DeclareGraphicsExtensions{.eps,.pdf,.png,.jpg}
\else
  \DeclareGraphicsExtensions{.eps}
\fi
   
\algnewcommand{\IIf}[1]{\State\algorithmicif\ #1\ \algorithmicthen}

\newcommand{\param}{\boldsymbol{\mu}}
\newcommand{\sol}{\boldsymbol{u}}
\newcommand{\fullstaten}{\textit{N}}
\newcommand{\gencoord}{\boldsymbol{y}}
\newcommand{\RR}{\mathbb{R}}
\newcommand{\nbasisspace}{\textit{k}}
\newcommand{\res}{\boldsymbol{r}}
\newcommand{\massmatrix}{\boldsymbol{A}}
\newcommand{\ffunction}{\boldsymbol{f}}
\newcommand{\gfunction}{\boldsymbol{g}}
\newcommand{\nmjacobian}{\jacobian_{\gfunction}}
\DeclareMathOperator*{\argmin}{\arg\!\min}
\DeclareMathOperator*{\argmax}{\arg\!\max}

\newcommand{\approxsol}{\widetilde{\sol}}

\newcommand{\approxres}{\widetilde{\res}}
\newcommand{\redres}{\widehat{\res}}
\newcommand{\approxf}{\widetilde{\ffunction}}

\newcommand{\basis}{\boldsymbol{\Phi}}
\newcommand{\basiscol}{\boldsymbol{\phi}}
\newcommand{\sourcebasis}{\boldsymbol{\Phi_f}}
\newcommand{\projmatrix}{\boldsymbol{\Psi}}
\newcommand{\projcol}{\boldsymbol{\psi}}
\newcommand{\resbasis}{\boldsymbol{\Phi_r}}

\newcommand{\nsamp}{\textit{n}_s}
\newcommand{\nbasissource}{\textit{n}_f}
\newcommand{\nbasisres}{\textit{n}_r}
\newcommand{\samplingmat}{\boldsymbol{Z}}
\newcommand{\sampset}{\mathcal{Z}} 
\newcommand{\sampneighs}{\sampset'} 

\usepackage{accents}

\newcommand{\gappyf}{\widehat{\ffunction}}
\newcommand{\gappyr}{\widehat{\res}}

\newcommand{\gappyerr}{\varepsilon}

\newcommand{\energyvar}{\textit{E}}
\newcommand{\convflux}{\boldsymbol{F_c}}
\newcommand{\viscflux}{\boldsymbol{F_v}}

\newcommand{\heatflux}{\boldsymbol{q}}

\newcommand{\solDomainSymbol}{\Omega}
\newcommand{\gradientSymbol}{\nabla}
\newcommand{\energySymbol}{e}
\newcommand{\velocitySymbol}{v}
\newcommand{\positionSymbol}{x}
\newcommand{\densitySymbol}{\rho}
\newcommand{\pressureSymbol}{p}
\newcommand{\stressSymbol}{\sigma}
\newcommand{\artificialStressSymbol}{\stressSymbol_a}
\newcommand{\adiabaticIndexSymbol}{\gamma}
\newcommand{\normalSymbol}{n}
\newcommand{\neumannSymbol}{g}
\newcommand{\identitySymbol}{I}

\newcommand{\timeSymbol}{t}
\newcommand{\jacobianSymbol}{J}
\newcommand{\massMatSymbol}{M}

\newcommand{\oneSymbol}{1}
\newcommand{\forceSymbol}{F}

\newcommand{\kinematicSymbol}{V}
\newcommand{\thermodynamicSymbol}{E}
\newcommand{\feSpace}[1]{\mathcal{#1}}

\newcommand{\sizeFOMsymbol}{N}

\newcommand{\timeIndex}{n}

\newcommand{\ROMBasisSymbol}{\mathbf{\Phi}}

\newcommand{\ROMSamplingMatSymbol}{\mathbf{Z}}

\newcommand{\relErrorSymbol}{\mathcal{\epsilon}}

\newcommand{\fullNorm}[1]{{\left\vert\kern-0.25ex\left\vert\kern-0.25ex\left\vert #1 
    \right\vert\kern-0.25ex\right\vert\kern-0.25ex\right\vert}}
\newcommand{\stateSymbol}{w}

\newcommand{\forceOne}{\forceSys_{\mathfrak{1}}}
\newcommand{\forceTv}{\forceSys_{t\velocitySymbol}}
\newcommand{\avgforceTv}{\bar{\forceSys}_{t\velocitySymbol}}
\newcommand{\forceOnek}[1]{\forceOne^{#1}}
\newcommand{\forceTvk}[1]{\forceTv^{#1}}
\newcommand{\avgforceTvk}[1]{\avgforceTv^{#1}}

\newcommand{\energy}{\boldsymbol{\energySymbol}}
\newcommand{\velocity}{\boldsymbol{\velocitySymbol}}
\newcommand{\position}{\boldsymbol{\positionSymbol}}
\newcommand{\state}{\boldsymbol{\stateSymbol}}

\newcommand{\velocityApprox}{\tilde{\velocity}}
\newcommand{\energyApprox}{\tilde{\energy}}
\newcommand{\positionApprox}{\tilde{\position}}

\newcommand{\relError}{\relErrorSymbol}

\newcommand{\massMat}{\boldsymbol{\massMatSymbol}}

\newcommand{\oneVec}{\boldsymbol{\oneSymbol}}
\newcommand{\kinematicFE}{\feSpace{\kinematicSymbol}}
\newcommand{\thermodynamicFE}{\feSpace{\thermodynamicSymbol}}
\newcommand{\kinematicMassMat}{\massMat_{\kinematicFE}}
\newcommand{\thermodynamicMassMat}{\massMat_{\thermodynamicFE}}

\newcommand{\jacobian}{\boldsymbol{\jacobianSymbol}}

\newcommand{\forceMat}{\boldsymbol{\forceSymbol}}

\newcommand{\forceSys}{\mathsf{\forceSymbol}}

\newcommand{\sizeWholeFE}{\sizeFOMsymbol}

\newcommand{\timek}[1]{\timeSymbol_{#1}}
\newcommand{\finalTime}{T}
\newcommand{\ntimestep}{N_{\timeSymbol}}
\newcommand{\timestep}{\Delta\timeSymbol}
\newcommand{\timestepk}[1]{\timestep_{#1}}

\newcommand{\avgvelocityt}[1]{\bar{\velocity}^{#1}}
\newcommand{\statet}[1]{\state^{#1}}
\newcommand{\energyt}[1]{\energy^{#1}}
\newcommand{\velocityt}[1]{\velocity^{#1}}
\newcommand{\positiont}[1]{\position^{#1}}

\newcommand{\energyApproxt}[1]{\energyApprox^{#1}}
\newcommand{\velocityApproxt}[1]{\velocityApprox^{#1}}
\newcommand{\positionApproxt}[1]{\positionApprox^{#1}}

\newcommand{\forceTvBasis}{\ROMBasisSymbol_{\forceTv}}
\newcommand{\forceOneBasis}{\ROMBasisSymbol_{\forceOne}}

\newcommand{\forceTvSamplingMat}{\ROMSamplingMatSymbol_{\forceTv}}
\newcommand{\forceOneSamplingMat}{\ROMSamplingMatSymbol_{\forceOne}}

\newcommand{\ntimestepROM}{\widetilde{N}_{\timeSymbol}}

\newcommand{\relerrorVelocityt}[1]{\relError_{\velocitySymbol,#1}}
\newcommand{\relerrorEnergyt}[1]{\relError_{\energySymbol,#1}}
\newcommand{\relerrorPositiont}[1]{\relError_{\positionSymbol,#1}}

\newsiamremark{remark}{Remark}
\newsiamremark{hypothesis}{Hypothesis}
\crefname{hypothesis}{Hypothesis}{Hypotheses}
\newsiamthm{claim}{Claim}

\title{S-OPT: 
A Points Selection Algorithm for 
Hyper-Reduction in Reduced Order Models\thanks{Submitted to the editors March 12, 2022.}}

\author{
  Jessica T. Lauzon\thanks{Department of Aeronautics and Astronautics, Stanford University, Stanford, CA 94350 (\email{jlauzon@stanford.edu})}
  \and 
  Siu Wun Cheung\thanks{Center for Applied Scientific Computing, Lawrence
  Livermore National Laboratory, Livermore, CA 94550 (\email{cheung26@llnl.gov}, \email{choi15@llnl.gov}, \email{copeland11@llnl.gov})}
  \and
  Yeonjong Shin\thanks{Division of Applied Mathematics, Brown University, Providence, RI 02912
  (\email{yeonjong\_shin@brown.edu})}
  \and
  Youngsoo Choi\footnotemark[3]
  \and
  Dylan Matthew Copeland\footnotemark[3]
  \and
  Kevin Huynh\thanks{Applications, Simulations, and Quality, Lawrence Livermore
  National Laboratory, Livermore, CA 94550 (\email{huynh24@llnl.gov})}
}

\usepackage{amsopn}

\ifpdf
\hypersetup{
  pdftitle={Sampling Algorithm Title},
  pdfauthor={J. Lauzon }
}
\fi

\begin{document}

\maketitle

\begin{abstract}
  While
  projection-based reduced order models 
  can reduce 
  the dimension of full order solutions, 
  the resulting reduced models may still contain terms that scale with the full order dimension.
  Hyper-reduction techniques are sampling-based methods that further reduce 
  this computational complexity 
  by approximating such terms 
  with a much smaller dimension.
  The goal of this work is to introduce 
  a points selection algorithm developed by Shin and Xiu [\textit{SIAM J. Sci. Comput.}, 38 (2016), pp. A385--A411],  
  as a hyper-reduction method.
  The selection algorithm 
  is originally proposed as a stochastic collocation method for uncertainty quantification.
  Since the algorithm aims at maximizing a quantity $\mathcal{S}$
  that measures both the \textit{column orthogonality} and the \textit{determinant}, 
  we refer to the algorithm as S-OPT.
  Numerical examples are provided to demonstrate the performance of S-OPT and to compare its performance with an over-sampled Discrete Empirical Interpolation (DEIM) algorithm.
  We found that using the S-OPT algorithm is shown to predict the full order solutions 
  with higher accuracy for a given number of indices. 
\end{abstract}



\section{Introduction}
Physical simulation is the key to the developments of science, engineering and
technology. Various physical processes are mathematically modeled by time-dependent nonlinear partial differential equations (PDEs). 
Since analytical solutions to such problems are not available in general, one has to resort to numerical methods to effectively approximate the solution.
State-of-the-art numerical methods 
have been proven successful in obtaining accurate approximations of 
the groundtruth observations in various application problems.
However, subject to the complexity and the scale of the problem domain, 
the computational cost of such numerical methods could be prohibitively high.
Even with high-performance computing,
a single forward simulation could take a very long time.
Yet, multiple forward simulations are typically required in 
some real world decision-making applications 
such as design optimization \cite{wang2007large, de2020three,
de2018adaptive, white2020dual}, optimal control \cite{choi2015practical,
choi2012simultaneous}, uncertainty quantification \cite{smith2013uncertainty,
biegler2011large}, and inverse problems \cite{galbally2010non,
biegler2011large}, 
which make such problems computationally intractable.

Constructing a reduced order model (ROM) is a popular and powerful computational technique 
to obtain sufficiently accurate numerical solutions 
with considerable speed-up
compared to the corresponding full order model (FOM).
Various model reduction schemes have been proposed.
Many of them seek to extract an intrinsic solution manifold using 
a condensed solution representation.
Depending on how these representations are constructed, two major approaches exist--
linear subspace reduced order models (LS-ROM) 
and nonlinear manifold reduced order models (NM-ROM).
In either case, the governing equations are projected onto the solution manifold
as part of the reduction strategy, and therefore these approaches are referred to as 
projection-based reduced order models (PROMs). 
This differs from other ROM techniques such as interpolation or data fitting.

In LS-ROM, the reduced basis vectors are obtained through, for example, 
proper orthogonal decomposition (POD).
The number of degrees of freedom is then reduced by substituting the ROM solution representation into the (semi-)discretized governing equation.  
This approach takes advantage of both the known governing equations and the solution data
generated from the corresponding FOM simulations to form LS-ROM.  
Example applications include, but are not limited to, nonlinear diffusion equations \cite{hoang2020domain, fritzen2018algorithmic}, Burgers equation and the Euler equations in small-scale \cite{choi2019space,
choi2020sns, carlberg2018conservative} and large-scale, convection--diffusion equations \cite{mojgani2017lagrangian, kim2021efficient}, the Navier--Stokes equations
\cite{xiao2014non, burkardt2006pod}, the compressible Euler equations in a 
moving Lagrangian frame \cite{copeland2022reduced, cheung2022local}, rocket nozzle shape design
\cite{amsallem2015design}, flutter avoidance wing shape optimization \cite{choi2020gradient}, topology optimization of wind turbine blades
\cite{choi2019accelerating}, lattice structure design \cite{mcbane2021component}, porous media flow/reservoir simulations \cite{ghasemi2015localized, jiang2019implementation, yang2016fast,
wang2020generalized}, computational electro-cardiology \cite{yang2017efficient},
inverse problems \cite{fu2018pod}, shallow water equations \cite{zhao2014pod, cstefuanescu2013pod}, Boltzmann transport problems \cite{choi2021space},
computing electromyography \cite{mordhorst2017pod}, spatio-temporal dynamics of a predator--prey system \cite{dimitriu2013application}, acoustic wave-driven microfluidic biochips \cite{antil2012reduced}, and the Schr{\"o}dinger equation \cite{cheng2016reduced}. 
However, in advection-dominated problems, the intrinsic solution space cannot be approximated by subspaces with a small dimension, i.e., the solution space with slowly decaying Kolmogorov $n$-width. As an alternative to LS-ROM, we can replace the linear subspace solution representation with a nonlinear manifold. This type of ROM is known as NM-ROM. 
A neural network-based reduced order model is developed in \cite{lee2020model} and extended to preserve the conserved quantities in the physical conservation laws \cite{lee2019deep}.  
Recently, Kim, et al., \cite{kim2022fast, kim2020efficient} have achieved a considerable speed-up with NM-ROMs via autoencoder. 

The main advantage of a ROM (either LS-ROM or NM-ROM) is to reduce the computational cost by using a low-dimensional structure for representation of state variables.
However, in nonlinear systems of PDEs,
the performance of ROM is degraded due to the bottleneck issue of lifting to FOM size.
That is, the nonlinear terms need to be evaluated in every time step as the state variables evolve in the time marching process. 
Since such evaluation scales with the FOM size, we cannot expect any speed-up without special treatment even if reduced representation is used to approximate the state variables. 
To overcome this issue, a hyper-reduction technique \cite{ryckelynck2005priori} is used to efficiently evaluate the nonlinear source terms by approximation.
We note that \textit{a majority of the aforementioned literature achieved 
a true speed-up by applying a hyper-reduction technique.}
The key idea is to approximate the nonlinear terms using a small number of basis vectors,
while keeping the number of evaluations of nonlinear terms as small as possible.
To this end, the hyper-reduction technique requires one to strategically select 
a set of indices that leads to accurate approximations to the nonlinear terms.

One of the most well-known selection algorithms is the Discrete Empirical Interpolation Method (DEIM) 
\cite{chaturantabut2010nonlinear},
which aims at minimizing the operator norm from the error estimate (Lemma 3.2 of \cite{chaturantabut2010nonlinear}).
DEIM is implemented through a greedy algorithm that sequentially selects one index at a time with respect to a certain criterion. 
Carlberg, et al., \cite{carlberg2013gnat} and \cite{carlberg2011efficient} 
extend this idea to allow oversampling. 
Q-DEIM is introduced in \cite{drmac2016new} as a new framework for constructing the 
DEIM-related operator via the QR factorization with column pivoting. 
The stability and oversampling of DEIM is also investigated in \cite{peherstorfer2020stability}.

The goal of this paper is to introduce the S-OPT sampling method \cite{shin16} 
as a hyper-reduction technique in ROMs,
and compare its performance with those by the DEIM family algorithms.
S-OPT was first developed by Shin and Xiu in \cite{shin16} as a points selection algorithm for least-squares based stochastic-collocation methods in Uncertainty Quantification (UQ) \cite{shin2016near,narayan2017christoffel,hadigol2018least,guo2020constructing}.
In the context of UQ, the goal is to find the best subset of points that yields the most accurate least-squares solution, which is closely aligned with the goal of the hyper-reduction. 
S-OPT aims to find a set of indices (rows or points) that maximizes the $\mathcal{S}$ quantity \cite{OptimalDesign} (to be introduced in Section~\ref{sec:S-OPT}).
The $\mathcal{S}$ quantity measures both the mutual column orthogonality and the determinant. 
The S-OPT algorithm is fundamentally different from the DEIM algorithm. 
The core principle of DEIM lies at maximizing \textit{the smallest singular value} (spectral norm) of the underlying projection matrix, 
while S-OPT seeks to maximize both \textit{the product of all the singular values} (determinant) and \textit{the column orthogonality}
of the underlying projection matrix.
We employ the S-OPT algorithm as an index selection operator for hyper-reduction.

\subsection{Paper organization}
The FOM is described in Section~\ref{sec:fom} to introduce some background information and notation. We then describe the PROM formulation in Section~\ref{sec:PROM}, 
which leads to the description of the hyper-reduction procedure in Section~\ref{sec:HR}. 
The sampling algorithms are part of the hyper-reduction procedure, 
so they are described in Section~\ref{sec:salgs}. 
Specifically, a DEIM algorithm with oversampling is outlined in Section~\ref{sec:greedy},
and the S-OPT algorithm is described in Section~\ref{sec:S-OPT}.
The performance and comparison of the two algorithms are presented in Section~\ref{sec:main} 
using four examples: 1D Burgers problem, 2D laminar viscous flow around airfoil, 
and two hydrodynamics examples, i.e., a 2D Gresho vortex problem and a 3D Sedov blast problem. The paper is concluded with summary and discussion in Section~\ref{sec:conclusions}.

\section{Problem Formulation}
We start by defining the FOM and some notation used throughout the paper. We use two different PROMs in the example problems. First, we outline the LS-ROM formulation, followed by NM-ROM formulation for a general ordinary differential equation (ODE). This section precedes the main contribution of this paper, which is a sampling algorithm used with both of the PROM formulations.

For the rest of the paper, $\|\cdot \|$ is understood as either the standard Euclidean norm or the spectral matrix norm.

\subsection{Full Order Model} \label{sec:fom}
Consider a system of nonlinear ODEs resulting from the semidiscretization of a system of PDEs in the space domain
 \begin{equation} \label{eq:fom}
  \massmatrix(\param) \dot{\sol}(t; \param)
  = \ffunction(\sol(t; \param),t, \param),\quad\quad
  \sol(0; \param) = \sol^0(\param),
 \end{equation} 
 where $t \in [0,\finalTime]$ denotes time, 
 $\sol(t; \param) \in \RR^\fullstaten$ denotes the state vector of dimension $\fullstaten$, 
 $\sol^0(\param) \in \RR^\fullstaten $ denotes the initial condition,
 $\param \in \mathcal{D}$ denotes a vector of parameters defining the operating point of interest 
 within the parameter domain $\mathcal{D} \subseteq \RR^{\fullstaten_\mu}$, 
 $\boldsymbol{A}(\param) \in \RR^{\fullstaten \times \fullstaten}$ denotes a nonsingular matrix, 
 and $\boldsymbol{f}$: $\RR^\fullstaten \times \RR \times \mathcal{D} \to \RR^\fullstaten$ 
 is a nonlinear function and boundary conditions. The parameter and time dependence of $\massmatrix$ 
 and variables are dropped in the rest of the paper for notational simplicity, and are implied.
 Furthermore, the dot notation denotes the derivative with respect to time.
 The FOM system in Eq.~\eqref{eq:fom} can be written in a residual form as follows: 
  \begin{equation} \label{eq:resfom}
  \res(\sol,\dot{\sol}, t, \param) := \massmatrix \dot{\sol} - \ffunction(\sol,t, \param) = 0.
 \end{equation} 
 
 The time derivative term above can be approximated by various time integration schemes. 
 Suppose the temporal domain $[0, \finalTime]$ is partitioned by $\{
  \timek{\timeIndex} \}_{\timeIndex=0}^{\ntimestep}$, where 
  $\ntimestep$ is the number of subintervals, $\timek{\timeIndex}$ denotes
a discrete moment in time with $\timek{0} = 0$, $\timek{\ntimestep} =
\finalTime$, and $\timek{\timeIndex-1} < \timek{\timeIndex}$ for 
$\timeIndex\in\{1,\ldots,\ntimestep\}$. Throughout the paper, we 
use a superscript $\timeIndex$ to denote the time-discrete counterpart
of a function evaluated at $t = \timek{\timeIndex}$. 
 For the numerical experiments, we use the implicit Backward Euler (BE), method and the second order explicit Runge--Kutta average (RK2-average) method to numerically solve Eq.~\eqref{eq:fom}, but other numerical time integration schemes are also applicable. For example, the BE method solves for $\sol^n$ at the $n$-th time step in Eq.~\eqref{eq:implicitfom}:
 \begin{equation} \label{eq:implicitfom}
  \massmatrix \sol^n - \massmatrix \sol^{n-1} = \Delta t \ffunction^n,
 \end{equation}
 where $\ffunction^n = \ffunction(\sol^{n},t_n, \param)$
 and $t_n$ is the $n$-th time.
 The residual function with the BE time integrator is then defined by
 \begin{equation} \label{eq:implicitres}
  \res^n_{\text{BE}}(\sol^n; \sol^{n-1}, \param) := \massmatrix (\sol^n - \sol^{n-1}) - \Delta t \ffunction^n.
 \end{equation}
 Although we continue the discussion using the BE time integrator and its residual as an example of the demonstration, the residuals of other types of time integrators can replace $\res^n_{\text{BE}}$ in a similar fashion. For example, we refer to \cite{copeland2022reduced} for the residual of the RK2-average method.

\subsection{Projection-Based Reduced Order Model}\label{sec:PROM}
 A PROM
 formulation relies on the concept that 
 full state solutions can be represented in lower-dimensional manifold. 
 As such, a PROM projects the governing equations to a manifold, resulting in lower-dimensional equations. If the manifolds for the solution field and the equations are the same, then the projection is called Galerkin. On the other hand, if the manifolds for the solution field and the equations are different, then the projection is called Petrov--Galerkin.  
 Both the Galerkin and Petrov--Galerkin projection methods are considered.
 We also consider a PROM with a linear subspace solution representation (LS-ROM), 
 as well as a model with nonlinear manifold solution representation (NM-ROM).

\subsubsection{Linear Subspace Reduced Order Model}\label{sec:LS-ROM}
A LS-ROM
reduces the spatial dimension by approximating the full solution 
using a subspace 
$\mathcal{W} := \text{span}\{\basiscol_i \in \mathbb{R}^N : i=1,\dots,\nbasisspace\}$ 
with 
$\text{dim}(\mathcal{W}) = \nbasisspace \ll \fullstaten$, also called a \textit{trial} subspace. 
The approximation $\approxsol$ of the full solution is
  \begin{equation} 
  \label{eq:lsrom}
   \approxsol = \sol_{\text{ref}} + \basis \gencoord \in \mathbb{R}^N,
 \end{equation} 
 where 
 $\sol_{\text{ref}} \in \RR^\fullstaten$ is a reference state, 
 $\basis \in \RR^{\fullstaten \times \nbasisspace}$ denotes a basis matrix whose $i$-th column is 
 $\basiscol_i$, 
 and $\gencoord \in \RR^\nbasisspace$ is a vector of unknown generalized coordinates. 
 
 The reduced subspace, $\basis$, is commonly found using POD
 \cite{pod}, 
 which is related to principal component analysis (PCA) in statistical analysis 
 and the Karhunen-Lo\`eve expansion in stochastic analysis \cite{hotelling1933analysis, loeve1955}. 
 While we use POD to define the basis in this paper, the basis can be built using other options, such as Fourier modes. 
 In POD, a set of basis functions are built by performing a singular value decomposition (SVD) 
 over a solution snapshot matrix. These snapshots are based on solutions of the FOM,
 either steady-state solutions for multiple parameters, or time-varying solutions. 
 Additionally, the reference state, $\sol_{\text{ref}}$, can be found by taking the average of the collected snapshots. 
 
 Substituting $\tilde{\sol}$ of Eq.~\eqref{eq:lsrom} for the full solution in Eq.~\eqref{eq:fom} results in a system of equations with fewer unknowns,
 \begin{equation} \label{eq:sol-sub}
  \res(\sol_{\text{ref}} + \basis \gencoord, \basis \dot{\gencoord}, t,\param) := \massmatrix \basis \dot{\gencoord} - \ffunction(\sol_{\text{ref}} + \basis \gencoord,t, \param) = 0,
 \end{equation} 
where $\dot{\gencoord}$ denotes the time derivative of the generalized coordinate $\gencoord$. 
Since $\basis$ and $\sol_{\text{ref}}$ are fixed,
let
$\redres (\gencoord, \boldsymbol{v}, t,\param) := \res(\sol_{\text{ref}} + \basis \gencoord, \basis \boldsymbol{v}, t,\param)$.
Note that $\redres(\gencoord, \boldsymbol{v}, t,\param)$ is linear in $\boldsymbol{v}$.
By projecting the system of equations onto 
a \textit{test} subspace with basis matrix $\projmatrix \in \RR^{\fullstaten \times \nbasisspace}$ (which can be different than $\basis$) whose $i$-th column is $\projcol_i$, 
that is,
\begin{align*}
    \left\langle \projcol_i, \redres (\gencoord, \boldsymbol{v}, t,\param) \right\rangle = 0, 
    \qquad i = 1,\dots,k,
\end{align*}
we solve for $\boldsymbol{v}$,
which gives 
the governing equation for $\gencoord$
\begin{equation} \label{eq:proj}
  \dot{\gencoord} = (\projmatrix^\top\massmatrix \basis)^{-1} \projmatrix^\top \ffunction(\sol_{\text{ref}} + \basis \gencoord,t, \param).
\end{equation} %
%

 Equation~\eqref{eq:proj} corresponds to Galerkin projection when the test subspace is the same as the trial subspace, i.e. $\projmatrix = \basis$. 
 When the test subspace differs, we have Petrov--Galerkin projection. However, Petrov--Galerkin projection is generally applied after discretizing
 in time, which leads us to describe the nonlinear Least-Squares Petrov--Galerkin (LSPG) projection procedure.
 A LSPG ROM substitutes $\approxsol^n = \sol_{\text{ref}} + \basis \gencoord^n$ for $\sol$
 into Eq.~\eqref{eq:implicitres} and minimizes the residual at each time instance. Using the BE time
 discretization as an example, we have
 \begin{equation} \label{eq:lspg}
     \gencoord^n = \argmin_{\boldsymbol{v} \in \RR^\nbasisspace} 
     \left\| \redres_{\text{BE}}^n (\boldsymbol{v}; \gencoord^{n-1}, \param ) \right\|^2,
 \end{equation}
where the LS-ROM backward Euler reduced residual is defined as
 \begin{equation} \label{eq:rbe}
  \begin{split}
     \redres_{\text{BE}}^n (\boldsymbol{v}; \gencoord^{n-1}, \param ) 
     :=& \res_{\text{BE}}^n(\sol_{\text{ref}} + \basis \boldsymbol{v}; \sol_{\text{ref}} + \basis \gencoord^{n-1}, \param) \\
     =& \massmatrix \basis (\boldsymbol{v} - \gencoord^{n-1}) - \Delta t \ffunction(\sol_{\text{ref}} + \basis \boldsymbol{v},t_n,\param).
  \end{split}
 \end{equation}
%
The necessary first-order optimality condition for Eq.~\eqref{eq:lspg} is
 \begin{equation*}
     (\jacobian^n \basis)^\top \redres_{\text{BE}}^n = 0,
     \quad \text{where} \quad 
     \jacobian^n = \massmatrix - \Delta t \jacobian_{\ffunction(\cdot,t_n,\param)}(\sol_{\text{ref}} + \basis \gencoord^{n}).
 \end{equation*}
 Here $\jacobian_{\ffunction(\cdot,t_n,\param)} \in \RR^{\fullstaten \times \fullstaten}$ is the Jacobian of $\ffunction(\cdot,t_n,\param)$.
 This shows that 
Eq.~\eqref{eq:lspg} corresponds to a Petrov--Galerkin projection of the FOM equations with a trial subspace $\basis$
and a test subspace $\projmatrix = \jacobian^n \basis$, hence the name
LSPG
projection.

A note on notation: throughout this paper, the hat ($\hat{\res}$) is used for reduced variables that lie in a smaller subspace such as $ \RR^\nbasisspace$, 
and the tilde ($\tilde{\sol}$) is used for variables in $\RR^\fullstaten$ that approximate their non-accented counterparts.

\subsubsection{Nonlinear Manifold Reduced Order Model}\label{sec:nm-rom}
The LS-ROM relies on a linear subspace ($\basis$) for the solution manifold. 
In this section, we outline a ROM with nonlinear solution representation.
As a generalization of the linear subspace representation in \eqref{eq:lsrom}, 
NM-ROM seeks approximation of the full state solutions in a trial manifold by
  \begin{equation} \label{eq:nmrom}
   \approxsol = \sol_{\text{ref}} + \gfunction(\gencoord), 
 \end{equation} 
 where $\approxsol \in \RR^\fullstaten$ denotes an approximation of the full solution, 
 $\sol_{\text{ref}} \in \RR^\fullstaten$ denotes a reference state, 
 $\boldsymbol{g}: \RR^{\nbasisspace} \to \RR^{\fullstaten}$ denotes a nonlinear function, 
 and $\gencoord \in \RR^\nbasisspace$ denotes a vector of unknown latent variables. 
 The nonlinear function is found by neural network training as a decoder in an autoencoder architecture. 
 In this work, we adopt the shallow masked autoencoder as in \cite{kim2022fast}, 
 where the nonlinear function $\boldsymbol{g}$ is a scaled decoder with one single hidden layer 
 and a sparsity mask in the output layer. The trainable parameters, i.e. the weight and bias in the 
 decoder and the encoder networks, are optimized against the mismatch between the original training data 
 and the corresponding autoencoder output. The trained decoder and its Jacobian, $\nmjacobian \in \RR^{\fullstaten \times \nbasisspace}$,
 is used to formulate the system of equations. More precisely, the counterpart of \eqref{eq:sol-sub} is 
  \begin{equation*} 
  \res(\sol_{\text{ref}} + \gfunction(\gencoord), \nmjacobian(\gencoord) \dot{\gencoord}, t, \param) 
  := \massmatrix \nmjacobian(\gencoord) \dot{\gencoord} 
  - \ffunction(\sol_{\text{ref}} + \gfunction(\gencoord),t, \param) = 0. 
 \end{equation*} 
 The NM-ROM Galerkin projection is then given by  
  \begin{equation*} 
  \dot{\gencoord} = 
  \left( \nmjacobian(\gencoord)^\top \massmatrix \nmjacobian(\gencoord) \right)^{-1}
  \nmjacobian(\gencoord)^\top \ffunction(\sol_{\text{ref}} + \gfunction(\gencoord),t, \param).
 \end{equation*} 
Similarly, the NM-ROM LSPG projection is given by 
  \begin{equation*} 
  \gencoord^{n} = \argmin_{\boldsymbol{v} \in \RR^\nbasisspace} \left\| 
  \redres^n_{\text{BE}}(\boldsymbol{v}; \gencoord^{n-1}, \param) \right\|^2,  
 \end{equation*}
 where the NM-ROM backward Euler reduced residual is defined as 
 \begin{equation*}
 \begin{split}
     \redres^n_{\text{BE}}(\boldsymbol{v}; \gencoord^{n-1}, \param)
     :=& \res^n_{\text{BE}}(\sol_{\text{ref}} + \gfunction(\boldsymbol{v}); \sol_{\text{ref}} + \gfunction(\gencoord^{n-1}), \param) \\
      =& \massmatrix \nmjacobian \left(\gfunction(\boldsymbol{v}) 
      - \boldsymbol{g}(\gencoord^{n-1})\right) - \Delta t \ffunction(\sol_{\text{ref}} + \gfunction(\boldsymbol{v}),t_n, \param). 
 \end{split}
 \end{equation*}

A note on time integrators: although the BE method is used for the discussion of LS-ROM and NM-ROM, other numerical time integrators can be used in a similar fashion.

\section{Hyper-Reduction and Sampling}\label{sec:HR}
Due to the term $\ffunction$ still being nonlinear in the reduced subspace, 
the requirement to compute the Jacobian, and to use the Jacobian in a matrix-matrix 
multiplication for LSPG, the reduced, low-dimensional equations may still not be 
computationally more efficient than the FOM. However, the ROM
framework allows for a further approximation of the equations using a hyper-reduction strategy. 

Specifically, when Galerkin projection is utilized, the $(\projmatrix^\top\massmatrix \basis)^{-1} \projmatrix^\top$ 
part of Eq.~\eqref{eq:proj} is time-independent, and can be computed offline for each parameter $\param$. 
This leaves the nonlinear function $\ffunction(\approxsol,t, \param)$ 
to be calculated at every iteration. However, $\ffunction$ still scales with the full dimension.
Furthermore, for Petrov--Galerkin projection, the projection matrix $\projmatrix$ must be calculated
at every iteration because the Jacobian relies on the current state. 
In the case of NM-ROM,
both projection methods rely on the Jacobian. 
The full dimension
dependency in the Jacobian and the nonlinear function limits the speed-up performance of the PROM.

There are two ways hyper-reduction can be applied. Firstly, the nonlinear term is approximated 
so that it scales with the reduced dimension. This is applicable when the projection matrix can 
be computed offline, so approximating only the nonlinear term results in sufficient computational savings. Secondly, the residual is approximated instead of independently working with the nonlinear term.
This is especially favorable in cases where the nonlinear term may be difficult to access in the FOM solver,
such as for residual minimization solvers. 
Either way, the aforementioned terms are approximated by using only a carefully constructed index subset and 
omitting gaps for ignored entries, 
so this procedure 
is called gappy tensor approximation, as first introduced in \cite{everson1995}.

\textbf{Hyper-reduction on $\ffunction$.}
The nonlinear term $\ffunction \in \RR^\fullstaten$ can be approximated in a least-squares sense
from a gappy form using a reduced basis $\basis_{\ffunction} \in \RR^{\fullstaten \times \nbasissource}$ 
built for the nonlinear term and the set of sampled indices, $\sampset = \{i_1,\dots,i_n\}$.
Given a gappy tensor $\gappyf \in \RR^{\nbasissource}$, the approximation $\approxf$ of $\ffunction$ is:
\begin{align}
    \nonumber \ffunction &\approx \approxf:= \basis_{\ffunction} \gappyf, \quad \text{with} \\
    \gappyf &= \argmin_{\bm{a} \in \RR^{\nbasissource}} \| \samplingmat^\top (\basis_{\ffunction}\bm{a} - \ffunction) \|, \label{eq:gappyLS}
\end{align}
where $\samplingmat = [\boldsymbol{e}_{i_1}, \cdots, \boldsymbol{e}_{i_n}] \in \RR^{\fullstaten \times n}$ 
is a sampling matrix which contains $\nbasissource \le n \ll \fullstaten$ columns of the $\fullstaten \times \fullstaten$ identity matrix,
and $\boldsymbol{e}_i$ is the standard basis in $\mathbb{R}^N$.
In other words, $\samplingmat$ samples $\ffunction$ at only $n$ indices.
Choosing which indices to select is the subject of Section~\ref{sec:salgs}. 
The (minimum norm) solution to Eq.~\eqref{eq:gappyLS} is
$\gappyf = (\samplingmat^\top \sourcebasis)^{\dagger} \samplingmat^\top \ffunction$
%
where the superscript $\dagger$ denotes the Moore-Penrose pseudoinverse. 
A hyper-reduction of the nonlinear term $\ffunction$ 
is then given as
%
\begin{equation} \label{eq:hrnonlinearterm}
    \approxf = \sourcebasis (\samplingmat^\top \sourcebasis)^{\dagger} \samplingmat^\top \ffunction,
\end{equation}
which is termed as ``gappy POD" hyper-reduction.
We remark that, while $\sourcebasis \sourcebasis^\dagger$ is 
an orthogonal projection onto the column space of $\sourcebasis$, 
the sampling procedure introduces an oblique projection 
$\sourcebasis (\samplingmat^\top \sourcebasis)^{\dagger} \samplingmat^\top$.


The procedure that combines Galerkin projection and gappy POD for the nonlinear term is known as the discrete empirical 
interpolation method (DEIM). In the original DEIM paper \cite{chaturantabut2010nonlinear}, $\sourcebasis$ was computed by 
collecting snapshots of the nonlinear term and computing the POD. 
However, it has also been shown that solution snapshots can be used through the subspace relation in \cite{choi2020sns}, i.e.,  
$\sourcebasis = \massmatrix\basis$. 

One can consider a simple approximation
$\approxf = \samplingmat \samplingmat^\top \ffunction$
to $\ffunction$,
which only requires $\samplingmat$.
This is termed ``collocation." 
Even in this case, however, $\sourcebasis$ may still be used to build $\samplingmat$ during the offline phase, 
yet is not used during the online phase of the ROM. This is justified by the assumption that $\ffunction$ is well represented by the subspace spanned by $\sourcebasis$.

\textbf{Hyper-reduction on the residual term.}
Given $\bm{v}$, $t_n$, $\bm{y}^{n-1}$ and $\bm{\mu}$,
the nonlinear term $\bm{r}:= \redres_{\text{BE}}^n (\bm{v};\bm{y}^{n-1},\bm{\mu})$ (defined in Eq.~\eqref{eq:rbe})
is similarly approximated by 
a gappy tensor $\gappyr \in \RR^{\nbasisres}$ using a reduced basis 
$\resbasis \in \RR^{\fullstaten \times \nbasisres}$ built for the residual and a set of sampled indices, $\sampset=\{i_1,\dots,i_n\}$:
\begin{align}
    \nonumber
    \res &\approx \approxres := \resbasis \gappyr, \quad \text{with}   \\
    \gappyr &= \argmin_{\bm{a} \in \RR^{\nbasisres}} \| \samplingmat^\top (\resbasis \bm{a} - \res) \|, \label{eq:gappyLSres}
\end{align}
where $\samplingmat = [\bm{e}_{i_1}, \cdots, \bm{e}_{i_n}] \in \RR^{\fullstaten \times n}$ is the sampling matrix
constructed from $\sampset$.
The (minimum norm) solution to Eq.~\eqref{eq:gappyLSres} is 
$\gappyr = (\samplingmat^\top \resbasis)^{\dagger} \samplingmat^\top \res$
%
and the corresponding hyper-reduction of the residual term $\res$ is 
\begin{equation} \label{eq:res_approx}
    \approxres = \resbasis (\samplingmat^\top \resbasis)^{\dagger} \samplingmat^\top \res,
\end{equation}
%
which is the gappy-POD approach to hyper-reduction.
Similarly as before,
one can also use
the collocation approximation
$\approxres = \samplingmat \samplingmat^\top \res$.
The procedure that combines LSPG projection and gappy POD hyper-reduction for the residual term is called the Gauss-Newton with approximated tensors (GNAT) procedure \cite{carlberg2011efficient}.

\textbf{Error analysis of hyper-reduction approximation.}
In practice, the sampling matrix $\samplingmat$ is not built; rather, the selected indices are maintained along with 
the corresponding rows of $\sourcebasis, \resbasis, \ffunction$ and $\res$. 
Regardless of which approach is used,
the goal is to approximate a vector $\bm{b}$
of size $N$ (either $\bm{b} = \bm{f}$ or $\bm{r}$)
using a predefined basis $\bm{M} \in \mathbb{R}^{N\times p}$
(either $\bm{M} = \basis_{\bm{f}}$, $p=\nbasissource$ or $\bm{M} = \basis_{\bm{r}}$, $p=\nbasisres$).
Hence, we are faced with the following 
optimization problem: 
Find the optimal sampling matrix $\samplingmat^* \in \mathbb{R}^{N\times n}$ such that 
\begin{equation} \label{true-OPT}
    \samplingmat^* =\argmin_{\samplingmat} \|\bm{b} - \tilde{\bm{b}}(\samplingmat) \|,
\end{equation}
where $\tilde{\bm{b}}(\samplingmat):= \bm{M}\tilde{\bm{a}}(\bm{Z})$
and $\tilde{\bm{a}}(\samplingmat) = \argmin_{\bm{a}} \|\samplingmat^\top \bm{M} \bm{a} - \samplingmat^\top \bm{b}\|$.
In Theorem~\ref{thm:error-estimate},
we quantify and estimate
the oblique projection error 
$\|\bm{b} - \tilde{\bm{b}}(\samplingmat) \|$,
which becomes the theoretical basis of 
many existing sampling methods.

\begin{theorem} \label{thm:error-estimate}
    Let $\samplingmat \in \mathbb{R}^{N\times n}$
    be a sampling matrix 
    and $\bm{M} \in \mathbb{R}^{N\times p}$ be 
    a basis matrix of full rank with $p \le n \le N$.
    Let $\tilde{\bm{a}}({\samplingmat}) = \argmin_{\bm{a}} \|\samplingmat^{\top}\bm{M}\bm{a} - \samplingmat^{\top}\bm{b}\|$
    and let $\tilde{\bm{b}}(\samplingmat):= \bm{M}\tilde{\bm{a}}({\samplingmat})$.
    Suppose $\samplingmat^\top \bm{M}$ is of full rank.
    Then,
    \begin{equation} \label{error_equality}
        \|\bm{b} - \tilde{\bm{b}}(\samplingmat)\|^2
        = \|\text{proj}_{\bm{M}^\perp} \bm{b} \|^2 + \|\epsilon(\samplingmat)\|^2,
    \end{equation}
    where 
    $\bm{M}=\bm{QR}$ is a QR factorization of $\bm{M}$,
    $\text{proj}_{\bm{M}^\perp} \bm{b} := (\bm{I}-\bm{QQ}^\top)\bm{b}$
    is the projection of $\bm{b}$
    onto the orthogonal complement of the column space of $\bm{M}$,
    and 
    $\bm{\epsilon}(\samplingmat)$ is the solution
    to $\min_{\bm{\epsilon}} \| \samplingmat^\top \bm{Q \epsilon} - \samplingmat^\top \text{proj}_{\bm{M}^\perp} \bm{b} \|$
    satisfying 
    $
        ((\samplingmat^\top \bm{Q})^\top \samplingmat^\top \bm{Q})\bm{\epsilon} =
        (\samplingmat^\top \bm{Q})^\top 
        \samplingmat^\top \text{proj}_{\bm{M}^\perp} \bm{b}.
    $
    Furthermore,
    \begin{equation} \label{error_inequality}
        \|\bm{b} - \tilde{\bm{b}}(\samplingmat)\| \le \|(\samplingmat^\top \bm{Q})^\dagger\| \cdot \|\text{proj}_{\bm{M}^\perp} \bm{b}\|.
    \end{equation}
\end{theorem}
\begin{proof}
    Let $\bm{a}^* = \tilde{\bm{a}}(\bm{I})$.
    Observe that 
    the quantities $A_M, Q_M, y_M, S_m, \hat{\beta}_M, \hat{\beta}_m$
    stated in Theorem 3.1 of \cite{shin16}
    are $\bm{M}, \bm{Q}, \bm{b}, \samplingmat^\top, \bm{a}^*, \bm{a}_{\samplingmat}$, respectively,
    in the notation of the current paper.
    It then follows from 
    Theorem 3.1 of \cite{shin16}
    that $\bm{Ma}^* - \bm{M}\tilde{\bm{a}}(\samplingmat) = -\bm{Q\epsilon}(\samplingmat)$.
    Since $\bm{M}\bm{a}^* = \bm{QQ}^\top \bm{b}$,
    we have
    \begin{align*}
        \bm{b} - \tilde{\bm{b}}(\samplingmat)
        &= \bm{b} - \bm{M}\tilde{\bm{a}}(\samplingmat)
        = \bm{b} - \bm{M}\bm{a}^* + \bm{M}\bm{a}^* - \bm{M}\tilde{\bm{a}}(\samplingmat) 
        = (\bm{I}-\bm{QQ}^\top)\bm{b} - \bm{Q\epsilon}(\samplingmat),
    \end{align*}
    which gives \eqref{error_equality}.
    Furthermore, it follows from 
    $\bm{\epsilon}(\samplingmat) = (\samplingmat^\top \bm{Q})^\dagger \samplingmat^\top \text{proj}_{\bm{M}^\perp}\bm{b}$
    and the above equality
    that 
    \begin{align*}
        \bm{b} - \tilde{\bm{b}}(\samplingmat)
        = (\bm{I} - \bm{Q}(\samplingmat^\top \bm{Q})^\dagger \samplingmat^\top)\text{proj}_{\bm{M}^\perp}\bm{b}.
    \end{align*}
    Since $\bm{Q}(\samplingmat^\top \bm{Q})^\dagger \samplingmat^\top$
    is a projection operator,
    the proof is completed 
    by observing 
    $\|\bm{Q}\| = \|\samplingmat^\top\| = 1$.
\end{proof}

\begin{remark} \label{remark:1}
    Theorem~\ref{thm:error-estimate} shows that 
the optimization problem \eqref{true-OPT}
is equivalent to 
\begin{equation} \label{true-OPT2}
    \samplingmat^* = \argmin_{\samplingmat} \|\bm{\epsilon}(\samplingmat)\|,
\end{equation}
as $\|\text{proj}_{\bm{M}^\perp} \bm{b}\|$ is independent of the sampling matrix $\samplingmat$.
Yet, since the optimal sampling matrix
$\samplingmat^*$ requires 
the full knowledge of $\bm{b}$,
the true optimality is hardly achieved
in practice.
\end{remark}

\begin{remark}
    Theorem~\ref{thm:error-estimate} 
    requires neither the column orthonormality of 
    $\bm{M}$
    nor the number $n$ of samples being 
    the same as the number $p$ of columns of $\bm{M}$.
    Hence, it
    can be viewed as a generalization of
    Lemma 3.2 of \cite{chaturantabut2010nonlinear}.
    For a particular case of $n = p$,
    since $(\samplingmat^\top \bm{Q})^\dagger = (\samplingmat^\top \bm{Q})^{-1}$,
    the error bound of \eqref{error_inequality}
    becomes 
    Equation 3.8 of \cite{chaturantabut2010nonlinear}.
\end{remark}

\subsection{Sampling Algorithms}\label{sec:salgs}

One of the most popular algorithms for the construction of the sampling matrix $\samplingmat$
is based on the error \textit{bound} of \eqref{error_inequality}.
These algorithms aim at constructing $\samplingmat$
that makes $\|(\samplingmat^\top \bm{Q})^{\dagger}\|$
as small as possible. 
This may be viewed as E-optimality \cite{pukelsheim2006optimal}
in the optimal design community, which maximizes the smallest  nonzero eigenvalue
of $(\samplingmat^\top \bm{Q})^\top \samplingmat^\top \bm{Q}$.
The commonly used algorithms in the PROM community, 
such as DEIM \cite{chaturantabut2010nonlinear}, Q-DEIM \cite{drmac2016new}, 
follow this principle. 


In the UQ community,
a similar problem has been addressed
in the context of least squares based stochastic collocation \cite{shin16, shin2016near,narayan2017christoffel,hadigol2018least,guo2020constructing}.
In particular, \cite{shin16} proposed a method
based on the error \textit{equality} of \eqref{error_equality}.
Since the true optimum is not available (as also mentioned in Remark~\ref{remark:1}),
\cite{shin16} developed 
the so-called S-optimality 
\cite{OptimalDesign}
that maximizes 
both the column orthogonality 
of $\samplingmat^\top \bm{Q}$
and the determinant
$(\samplingmat^\top \bm{Q})^\top \samplingmat^\top \bm{Q}$. Maximizing the column orthogonality of $\samplingmat^\top \bm{Q}$ can increase the numerical stability, which has been an issue for E-optimality based sampling methods, such as DEIM.
In this paper, we refer to this method as ``S-OPT".

The goal of this paper is to introduce 
the S-OPT sampling method
to the PROM community 
as a hyper-reduction method,
and compare its performance 
with those by the DEIM family algorithms.

In what follows, 
the following notations are used.
Let $\bm{Q}$ be an orthogonal matrix obtained from a QR factorization of $\bm{M}$
if the columns of $\bm{M}$ are not orthonormal,
and let $\bm{Q}=\bm{M}$ otherwise.
Given a set of indices $\sampset = \{i_1,\dots,i_\ell\}$
and a vector $\basiscol \in \mathbb{R}^{N}$,
the $i$th component of $\basiscol$ is denoted by $\basiscol(i)$.


\subsubsection{Oversampled DEIM Algorithm} \label{sec:greedy}


The oversampled DEIM algorithm used in this paper differs from the original DEIM algorithm \cite{chaturantabut2010nonlinear}
and more closely follows the sampling method from the LSPG paper (Algorithm 5 in \cite{carlberg2011efficient}). 
Unlike the original DEIM, 
this allows for oversampling, i.e., selecting more samples than the number of POD modes, i.e., $n > p$.
We call the algorithm ``oversampled DEIM" in this paper, 
while it could be fairly named a greedy or gappy-POD algorithm, among other options.
This algorithm has also been modified for space-time LSPG in \cite{choi2019space}.




By closely following notation of \cite{washabaugh},
we present 
the pseudo-code for the oversampled DEIM in Algorithm \ref{alg:deim}. 
Starting with an appropriate basis $\bm{Q}=[\basiscol_1,\cdots,\basiscol_p]$
as an input, as well as the desired number of samples $n$, the algorithm selects indices based on a greedy method.
The first selection is simply the index of the largest absolute value entry of $\basiscol_1$.
The algorithm then loops over the rest of the columns of $\bm{Q}$, retaining all prior columns within
$\bm{Q}_{1:j} =[\basiscol_1,\cdots,\basiscol_j]$.
Each column of $\bm{Q}$ is approximated using a gappy POD reconstruction (line 8), and the index at which there is the largest error is selected and included in $\sampset$. 
This loop continues until the sampled set contains the desired number of unique indices.
If necessary, the set $\sampneighs$ of total indices (sampled and required neighboring nodes) is also populated according to the numerical scheme.
\begin{algorithm}
\setstretch{1.35}
\caption{Oversampled DEIM selection.} \label{alg:deim}
\textbf{Input:} Desired number of sampled indices $n$ where $p \leq n < \fullstaten$, 
  and an orthonormal basis $\bm{Q} = [\basiscol_1, \cdots, \basiscol_p]$ from $\bm{M}$.\\
\textbf{Output:} A set of $n$ indices $\sampset =  \{i_1, \dots, i_n \}$
\begin{algorithmic} [1]
 \State $\sampset = \{ i_1 \}$ where $i_1 = \texttt{argmax}_i |\basiscol_1(i)|$
 \State $n_{\text{iter}} =$ ceil$(\frac{n-1}{p-1})$ \{Determine the number of indices to choose for each iteration\}
 \For{ $j = 2: p$} 
   \State 
   $\bm{Q}_{1:j-1} = [\basiscol_1, \cdots, \basiscol_{j-1}]$ 
   \For{ $k = 1 : n_{\text{iter}}$}
     \State
     $\ell = 
     (j-2)n_{\text{iter}} + k$
     \State 
     Construct 
     $\samplingmat = [\boldsymbol{e}_{i_1}, \cdots, \boldsymbol{e}_{i_\ell}]$
     from $\sampset$
     \State Compute $\bm{\gappyerr} = \bm{Q}_{1:j-1} (\samplingmat^\top \bm{Q}_{1:j-1})^{\dagger} \samplingmat^\top \basiscol_{j}$ 
     \State $i_{\ell+1} = \texttt{argmax}_{i} \left| \basiscol_{j}(i) - \bm{\gappyerr}(i) \right|$ 
     \State $\sampset \leftarrow \sampset \cup \{  i_{\ell+1} \}$
     \IIf {$|\sampset| = n$} \Return 
    \EndFor
    \EndFor
   
\end{algorithmic}
\end{algorithm}

\subsubsection{S-OPT: A Quasi-Optimal Points Selection Algorithm} \label{sec:S-OPT}

This subsection introduces the sampling method (S-OPT) proposed in \cite{shin16}.
The underlying principle of S-OPT is to find 
a sampling matrix $\samplingmat$ that
makes 
$\bm{\epsilon}(\samplingmat)$ as small as possible. According to Theorem~\ref{thm:error-estimate},
$\bm{\epsilon}(\samplingmat)$ should 
satisfy 
\begin{equation} \label{moti:S-opt}
    ((\samplingmat^\top \bm{Q})\samplingmat^\top \bm{Q}) \bm{\epsilon} = (\samplingmat^\top \bm{Q})^\top \samplingmat^\top \text{proj}_{\bm{M}^\perp} \bm{b}.
\end{equation}
If the sampling matrix $\samplingmat$ is constructed to make $\samplingmat^\top \bm{Q}$ preserve the same column orthogonality as $\bm{Q}$, the right-hand side of \eqref{moti:S-opt}
becomes zero, 
leading to $\bm{\epsilon}(\samplingmat) = 0$.
However, this will only be the case 
when $\samplingmat$ is the identity matrix
with $n = N$
and in general, $\bm{\epsilon}(\samplingmat)$ will be nonzero.
With the goal of having a small $\bm{\epsilon}(\samplingmat)$,
the S-OPT \cite{shin16} 
seeks to 
(a) maximize the column orthogonality of $\samplingmat^\top \bm{Q}$
so that the right-hand side of \eqref{moti:S-opt}
is minimized,
and 
(b) maximize the determinant of 
$(\samplingmat^\top \bm{Q})^\top \samplingmat^\top \bm{Q}$
so that the nonzero solution of \eqref{moti:S-opt} is small.
A quantity denoted by $\mathcal{S}$ (the precise definition is given in \eqref{def:S})
that measures the mutual column orthogonality and the determinant was developed in \cite{shin16}.
The S-OPT aims at finding a sampling matrix $\samplingmat$ that maximizes the quantity $\mathcal{S}(\samplingmat^\top \bm{Q})$.

Since the true optimality 
$\samplingmat^*$ of \eqref{true-OPT2}
is not achievable in practice, 
the selection method based on 
the S-optimality \eqref{def:S} was termed ``quasi-optimal" in \cite{shin16}
and ``near-optimal/S-optimality" in \cite{shin2016near}.
Here, we simply refer to it as ``S-OPT."



Let $\bm{A}$ be an $N\times p$ matrix
and $\bm{\alpha}_i$ be its $i$th column.
Assuming $\|\bm{\alpha}_i\|\ne 0$ for all $i=1,\dots,p$,
$\mathcal{S}(\bm{A})$ is defined to be 
\begin{equation} \label{def:S}
    \mathcal{S}(\bm{A}) := \left( \frac{\sqrt{\det \bm{A}^\top \bm{A}}}{\prod_{i=1}^p \|\bm{\alpha}_i\|} \right)^{\frac{1}{p}} \in [0,1].
\end{equation}
It was shown in \cite{shin16} that
$\mathcal{S}(\bm{A}) = 1$ if and only if 
the columns of $\bm{A}$ are mutually orthonormal.
Hence, 
maximizing $\mathcal{S}$
enforces both
mutual column orthogonality 
and a larger determinant.

In the context of hyper-reduction, 
the S-OPT seeks the optimal sampling matrix $\samplingmat_{\text{S-OPT}}$ that
maximizes $\mathcal{S}$, i.e.,
\begin{equation*}
    \samplingmat_{\text{S-OPT}} = \argmax_{ \samplingmat }  \mathcal{S}(\samplingmat^\top \bm{Q}).
\end{equation*}
Solving the above optimization problem
requires one to compute multiple evaluations of 
$\mathcal{S}$.
The evaluation of $\mathcal{S}$, however, can be expensive 
as it requires
the computation of determinants.
However, \cite{shin16} presented 
an efficient way of evaluating $\mathcal{S}$
without computing determinants, 
based on a greedy algorithm 
and the Sherman–Morrison formula.
The pseudo-code for the S-OPT index selection is presented in Algorithm~\ref{alg:s-opt}
for the reader's convenience.
We refer to the original paper \cite{shin16} for full algorithmic details.

\begin{algorithm}
\setstretch{1.35}
\caption{S-OPT selection algorithm \cite{shin16}.} \label{alg:s-opt}
\textbf{Input:} Desired number of sampled indices $n \in \{p,\dots,N\}$, 
  and an orthonormal basis $\bm{Q} = [\basiscol_1, \cdots, \basiscol_p]$ from $\bm{M}$. \\
\textbf{Output:} A set of $n$ indices $\sampset =  \{i_1, \dots, i_n \}$
\begin{algorithmic} [1]
 \State $\sampset = \{ i_1 \}$ where $i_1 = \texttt{argmax}_i |\basiscol_1(i)|$
 \For{ $j = 2 : n$} 
   \State 
   $\bm{Q}_{1:k} = [\basiscol_1, \cdots, \basiscol_{k}]$ 
   where $k = \min\{j,p\}$
   \State $i_j = \texttt{argmax}_{\ell \in [N]\backslash \sampset} \mathcal{S}(\samplingmat_\ell^\top \bm{Q}_{1:k})$
   where $\samplingmat_\ell$
   is constructed from 
   $\sampset \cup \{\ell\}$
   \State $\sampset \leftarrow \sampset \cup \{i_j\}$
 \EndFor
\end{algorithmic}
\end{algorithm}

\section{Main results}
\label{sec:main}

In this section we show four example problems using both sampling algorithms. They are a 1D Burgers equation problem, an aerodynamic laminar flow around airfoil, and two Lagrangian hydrodynamic problems.

The 1D Burgers equation compares the use of the sampling methods for reduced order models built with linear or nonlinear subspaces; that is, LS-ROM and NM-ROM. Meanwhile, the laminar airfoil problem builds a parametric LS-ROM using shape parameters. The airfoil problem also explores how truncating the POD basis may affect the results when using either sampling algorithm. The final two problems are the 2D Gresho vortex and the 3D Sedov blast Lagrangian hydrodynamics problems. In these examples, we show that the sampling methods result in different error bounds, and we further examine the performance of the two sampling algorithms for varying the number of sampled indices. 

Furthermore, the examples in this paper use different combinations of projection and hyper-reduction methods.
Before presenting the problem descriptions and results, we present the specific methods used for each of the four examples here, for comparison and completeness purposes.

While the examples in Sections~\ref{sec:burg} and~\ref{sec:airfoil} both apply LSPG projection, 
the example in Section~\ref{sec:burg} applies gappy POD hyper-reduction
(or the GNAT procedure), and the airfoil example in Section~\ref{sec:airfoil} applies collocation hyper-reduction for the residual term.
Finally, the two examples in Section~\ref{sec:laghydro} use the DEIM approach; i.e., a
Galerkin projection with a gappy POD approximation of the nonlinear term. 

Applying LSPG with \textit{gappy POD} hyper-reduction, as done in Section~\ref{sec:burg}, gives
  \begin{equation} \label{eq:lspggappyhr}
  \gencoord^{n} = \argmin_{\boldsymbol{v} \in \RR^\nbasisspace} \left\| 
  (\samplingmat^\top \resbasis)^{\dagger} \samplingmat^\top
  \res^n_{\text{BE}}(\sol_{\text{ref}} + \basis \boldsymbol{v}; \sol_{\text{ref}} + \basis \gencoord^{n-1}, \param) \right\|^2, 
 \end{equation}
 which aims at minimizing the approximation \eqref{eq:res_approx} of $\res^n_{\text{BE}}$. 

 Applying LSPG with \textit{collocation} hyper-reduction, as done in Section~\ref{sec:airfoil}, gives
   \begin{equation} \label{eq:lspgcolhr}
  \gencoord^{n} = \argmin_{\boldsymbol{v} \in \RR^\nbasisspace} \left\| 
  \samplingmat^\top 
  \res^n_{\text{BE}}(\sol_{\text{ref}} + \basis \boldsymbol{v}; \sol_{\text{ref}} + \basis \gencoord^{n-1}, \param) \right\|^2.   
 \end{equation}
 In Equations~\eqref{eq:lspggappyhr} and \eqref{eq:lspgcolhr}, the fully discretized residual 
 $\res^n_{\text{BE}}$ is used as an example, and a residual defined using a different discretization can be used in a similar fashion.
 
 And lastly, applying Galerkin projection with gappy POD hyper-reduction of the nonlinear term, 
 as done in Section~\ref{sec:laghydro}, gives
  \begin{equation} \label{eq:galerkinhr}
  \begin{split}
      \dot{\gencoord}  
      &= (\projmatrix^\top \massmatrix \basis)^{-1} \projmatrix^\top \approxf
      \\
      &= (\projmatrix^\top \massmatrix \basis)^{-1} \projmatrix^\top 
      \left[ \sourcebasis (\samplingmat^\top \sourcebasis)^{\dagger} 
  \samplingmat^\top \ffunction(\sol_{\text{ref}} + \gfunction(\gencoord),t, \param)\right]. 
  \end{split}
 \end{equation} 
In the case of Galerkin projection, remember that the projection matrix $\projmatrix$ is taken to be the solution basis $\basis$. 
The front matter $(\projmatrix^\top \massmatrix \basis)^{-1}\projmatrix^\top \sourcebasis (\samplingmat^\top \sourcebasis)^{\dagger}$
can be pre-computed once, and the nonlinear term $\samplingmat^\top \ffunction$ can be evaluated at only the selected indices.
Equations~(\ref{eq:lspggappyhr}, \ref{eq:lspgcolhr}, \ref{eq:galerkinhr}) no longer scale with the full dimension $\fullstaten$, 
but rather with the number of selected indices $n$ and the reduced dimension $\nbasisspace$.

Finally, it is important to note that computing the residual at the selected nodes may depend on the solution at neighboring nodes. 
For example, for a finite volume discretization involving inviscid fluid dynamics (Euler) equations (e.g., Section~\ref{sec:airfoil}) as well as the finite element method (e.g., Section~\ref{sec:laghydro}), the solution at the 
neighboring nodes is required to compute the flux, and then the residual at the selected node can be computed using the flux.
For this reason it is necessary to maintain the indices of all required neighboring nodes, in addition to the selected nodes.

\subsection{1D Burgers Equation}\label{sec:burg}
As the first example, we consider the 1D inviscid Burgers equation 
\begin{equation}\label{eq:burgers}
  \begin{aligned}
    \frac{\partial u}{\partial\timeSymbol} + u \frac{\partial u}{\partial x} & = 0, 
    \quad (t,x) \in [0,0.5]\times[0,2],
  \end{aligned}
\end{equation}
with the periodic boundary condition $u(\timeSymbol, 0) = u(\timeSymbol, 2)$ for $t \in [0,0.5]$
and the initial condition $u(0,x) = 1 + \frac{1}{2}(\sin(2\pi x - \frac{\pi}{2})+1)$. 
We decompose the spatial domain into a uniform mesh with mesh size $\Delta x = 0.002$, 
which consists of $1001$ grid points $x_i = i \Delta x, i \in \{0,1,2,\ldots,1000\}$, 
at which the discrete solution is defined by $u_i(t) = u(t,x_i)$. 
The semi-discretization of the Burgers equation is given by: 
\begin{equation}\label{eq:burgers-fom}
  \begin{aligned}
    \frac{d\sol}{\partial\timeSymbol} = -\dfrac{1}{\Delta x}(\boldsymbol{D}\sol \odot \sol),
  \end{aligned}
\end{equation}
where $\sol = (u_1, u_2, \ldots, u_{1000})^\top \in \RR^{1000}$ 
is the coefficient vector, $\odot$ denotes the entry-wise product, and 
\begin{equation}
    \boldsymbol{D} = \left[\begin{array}{ccccc}
    1 & & & & -1 \\
    -1 & 1 & & & \\
     & -1 & 1 & & \\
     & & \ddots & \ddots & \\
     & & & -1 & 1
    \end{array}\right].
\end{equation}
To obtain a fully discrete scheme, 
we similarly decompose the temporal domain into $\ntimestep = 500$ subintervals
to form a uniform mesh with mesh size $\Delta t = 0.001$, 
and use a backward difference to numerically approximate the temporal derivative. 
Figure~\ref{fig:burgers_sol} shows the initial condition and the final-time solution 
for 1D Burgers equation. It can be seen that a shock is eventually developed at $x = 1.5$. 
\begin{figure}[htb!]
    \centering
    \includegraphics[width=0.48\linewidth]{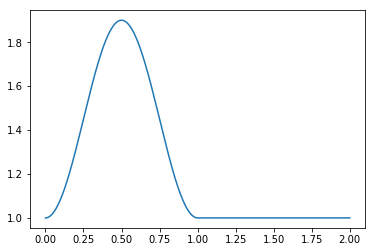}
    \hfill
    \includegraphics[width=0.48\linewidth]{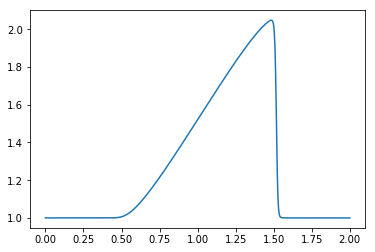}
    \caption{Initial condition (left) and final-time solution (right) for 1D Burgers equation.}
    \label{fig:burgers_sol}
\end{figure}
In \cite{kim2022fast}, a linear-subspace reduced order model (LS-ROM) and 
a nonlinear-manifold reduced order model (NM-ROM) are developed for \eqref{eq:burgers-fom}, where DEIM is used 
for hyper-reduction in both approaches. 
In what follows, we compare the numerical results of reduced order models 
using the sampling algorithms in Section~\ref{sec:salgs}, 
i.e. the oversampled DEIM and S-OPT. 
The maximum-in-time $L^2$ relative error for the ROM solution $\widetilde{\sol}$ is measured
against the corresponding FOM solution $\sol$,
which is defined as:
\begin{align}\label{eq:relerror-burger}
  \relError_{u,\max} &= \frac{ \max\limits_{1 \leq n \leq 500} \| \sol^n -
  \widetilde{\sol}^{n} \|_{L^2} }{ \max\limits_{1 \leq n \leq 500} \| \sol^n \|_{L^2} }.
\end{align}
All the 1D Burgers equation simulations in this
subsection use Lassen in Livermore Computing Center\footnote{High
performance computing at LLNL, https://hpc.llnl.gov/hardware/platforms/lassen},
on Intel Power9 CPUs with 256 GB memory and NVIDIA V100 GPUs, 
peak TFLOPS of 23,047.20, and peak single
CPU memory bandwidth of 170 GB/s.

In Figure~\ref{fig:burgers_error}, we depict the maximum-in-time relative discrete $L^2$ 
error of the solution against the number of sampling indices for both sampling algorithms. 
The number of sampling indices $\nsamp$ takes value within $\nbasissource \leq \nsamp \leq \nbasissource+30$, where $\nbasissource=30$ in LS-ROM and $\nbasissource=47$ in NM-ROM. 
In LS-ROM, despite only a minor difference in the error, the DEIM algorithm
yields more oscillation in the error with respect to the number of sampling indices. 
In NM-ROM, the oscillation is even more severe in DEIM. 
In other words, the S-OPT algorithm yields more stable results 
in the Burgers equation with increasing number of sampling indices. 
Figure~\ref{fig:burgers_lsrom_node} and Figure~\ref{fig:burgers_nmrom_node} 
illustrate the selected nodes by the sampling algorithms in LS-ROM and NM-ROM respectively. 
In both cases, both S-OPT and DEIM select nodes around the expended shock wave, 
where the nonlinearity occurs, in priority. As the oversampling number increases, 
S-OPT tends to sample nodes in a more widespread manner, 
while DEIM still tends to densely select nodes close to the shock, 
as shown in Figure~\ref{fig:burgers_sol}. 
\begin{figure}[htb!]
    \centering
    \includegraphics[height=0.4\linewidth]{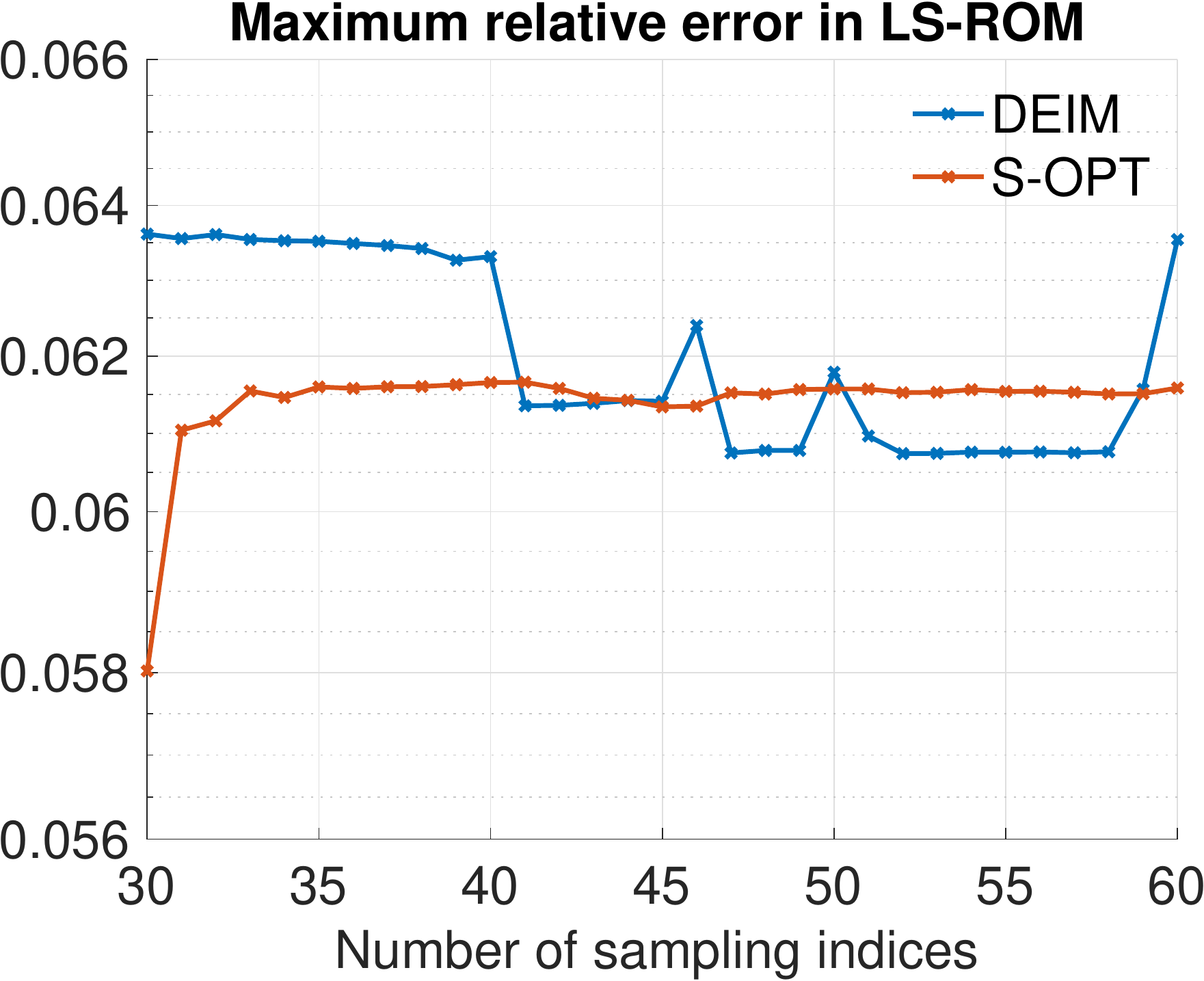}
    \hfill
    \includegraphics[height=0.4\linewidth]{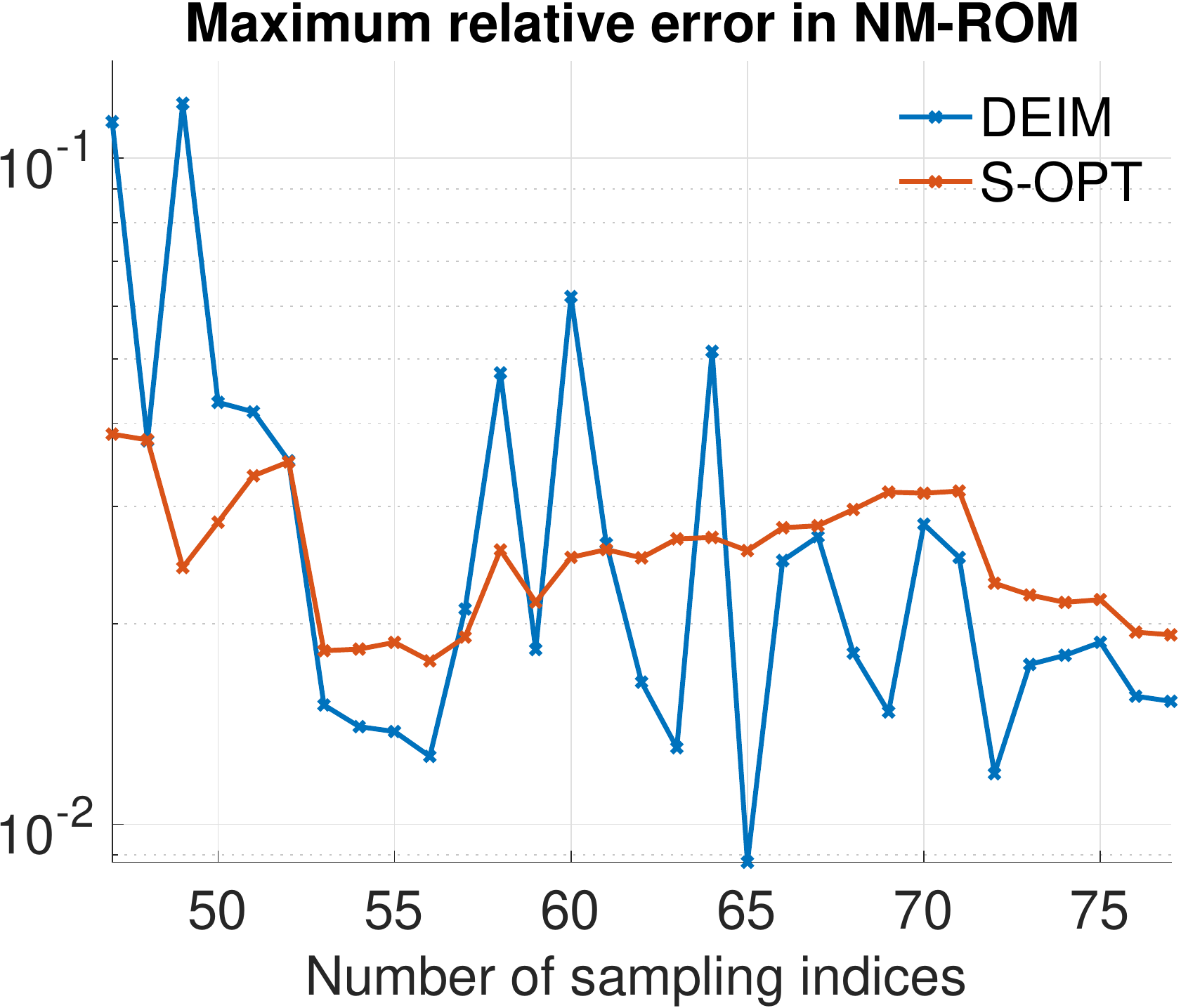}
    \caption{Maximum-in-time relative $L^2$ error with varying number of sampling indices in LS-ROM (left) and NM-ROM (right) for 1D Burgers equation.}
    \label{fig:burgers_error}
\end{figure}
\begin{figure}[htb!]
     \centering
     \begin{subfigure}[b]{0.45\textwidth}
         \centering
         \includegraphics[width=\textwidth]{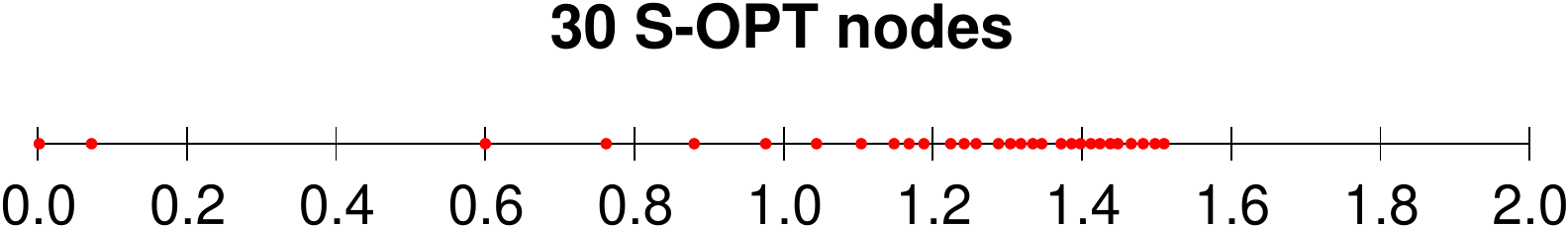}
     \end{subfigure}
     \hfill
     \begin{subfigure}[b]{0.45\textwidth}
         \centering
         \includegraphics[width=\textwidth]{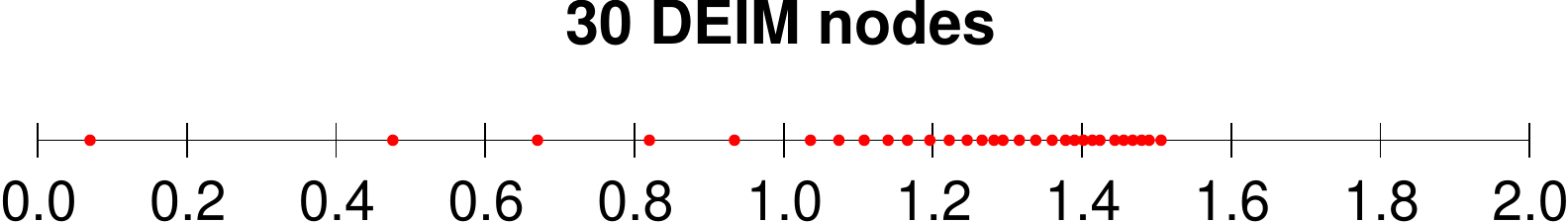}
     \end{subfigure} \\
     \vspace{0.2in}
     \begin{subfigure}[b]{0.45\textwidth}
         \centering
         \includegraphics[width=\textwidth]{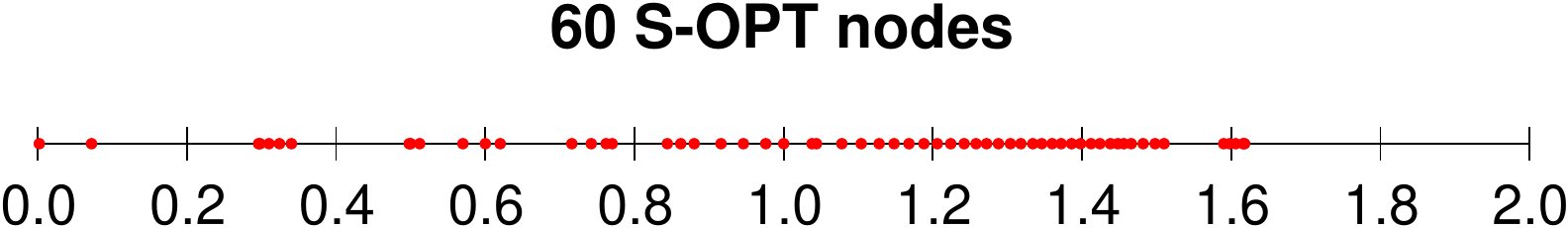}
     \end{subfigure}
     \hfill
     \begin{subfigure}[b]{0.45\textwidth}
         \centering
         \includegraphics[width=\textwidth]{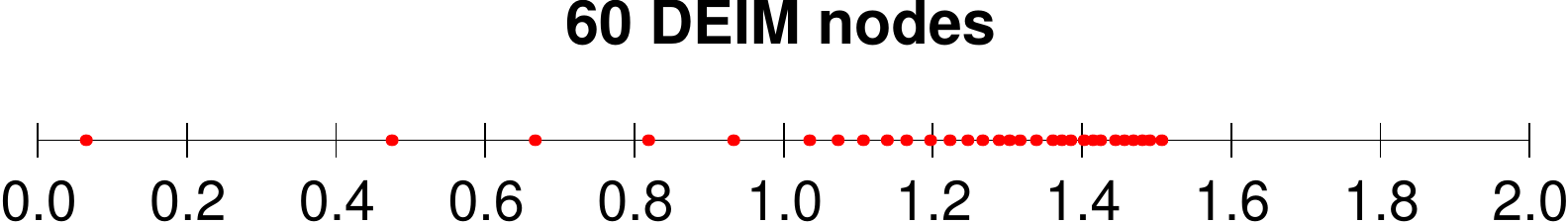}
     \end{subfigure}
        \caption{Selected nodes in LS-ROM for 1D Burgers equation.}
        \label{fig:burgers_lsrom_node}
\end{figure}
\begin{figure}[htb!]
     \centering
     \begin{subfigure}[b]{0.45\textwidth}
         \centering
         \includegraphics[width=\textwidth]{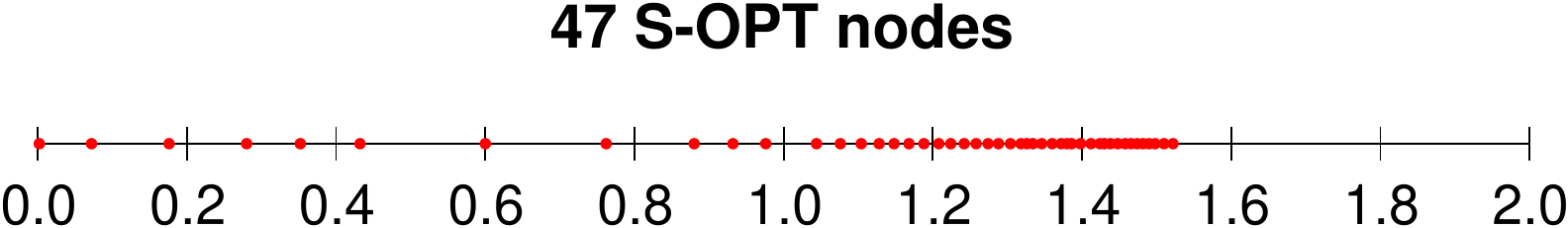}
     \end{subfigure}
     \hfill
     \begin{subfigure}[b]{0.45\textwidth}
         \centering
         \includegraphics[width=\textwidth]{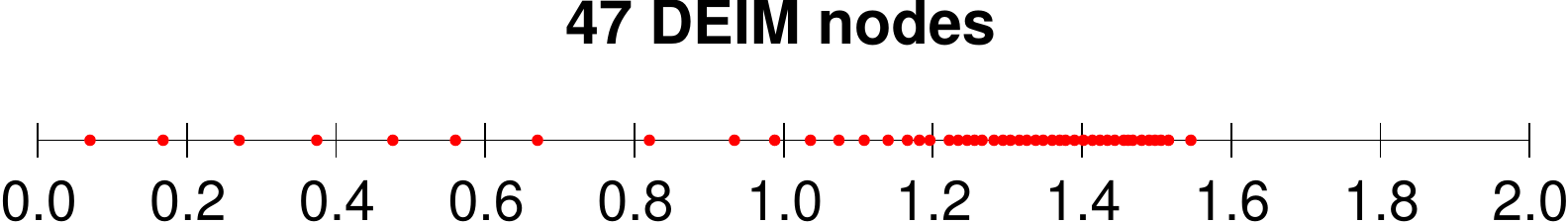}
     \end{subfigure} \\
     \vspace{0.2in}
     \begin{subfigure}[b]{0.45\textwidth}
         \centering
         \includegraphics[width=\textwidth]{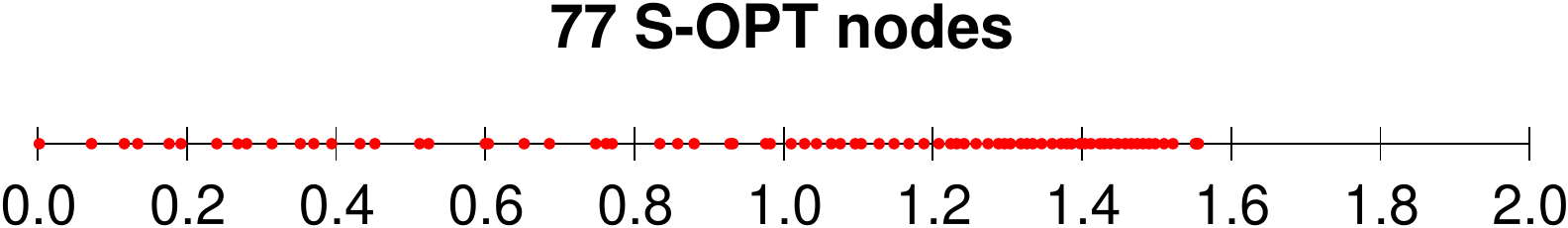}
     \end{subfigure}
     \hfill
     \begin{subfigure}[b]{0.45\textwidth}
         \centering
         \includegraphics[width=\textwidth]{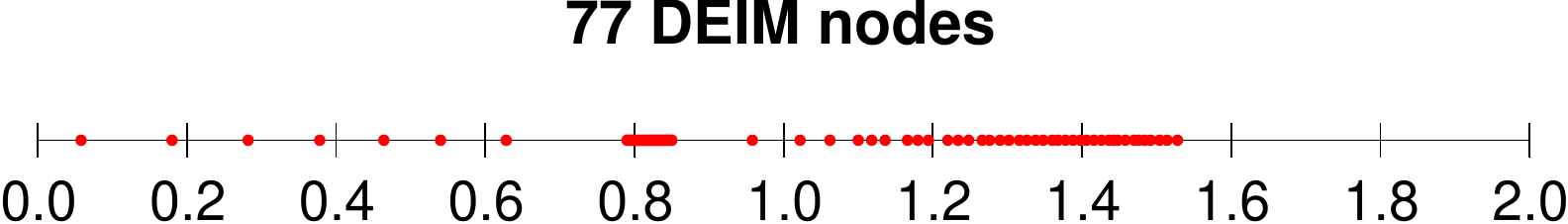}
     \end{subfigure}
        \caption{Selected nodes in NM-ROM for 1D Burgers equation.}
        \label{fig:burgers_nmrom_node}
\end{figure}
\subsection{2D Laminar Airfoil}\label{sec:airfoil}
The next example considers a 2D laminar airfoil using the steady compressible Navier-Stokes equations with a low Reynolds number (i.e., no turbulence). For the steady problem, we define the following residual definition, which is used in the construction of a LS-ROM:
\begin{equation}
    \res(\sol; \param) = 
    \gradientSymbol \cdot \convflux(\sol; \param) 
    - \gradientSymbol \cdot \viscflux(\sol; \param) = 0
    \quad \quad \text{in } \solDomainSymbol,
\end{equation}
where $\convflux$ is the convective flux, and $\viscflux$ is the viscous flux. 
For this 2D case, the conservative variables are given by 
$\sol = \{\densitySymbol, \densitySymbol \velocity, \rho \energyvar \}^\top$ 
with the velocity vector defined as $\velocity = \{\velocitySymbol_1, \velocitySymbol_2\}^\top \in \RR^2$, 
$\densitySymbol$ denoting the fluid density, and $\energyvar$ denoting the total energy per unit mass. 
Furthermore, $\convflux$ and $\viscflux$ are defined as:
\begin{equation}
    \small{\convflux} = \left\{ \begin{array}{c}
    \densitySymbol \velocity^\top \\ \densitySymbol \velocity \velocity^\top + \boldsymbol{I} \pressureSymbol \\ 
    \densitySymbol \energyvar \velocity^\top + \pressureSymbol \velocity^\top \end{array} \right\}, ~ ~
    \small{\viscflux} = \left\{ \begin{array}{c}
    . \\ \boldsymbol{\tau} \\ (\boldsymbol{\tau} \cdot \velocity)^\top + \heatflux^\top \end{array} \right\},
\end{equation}
where $\pressureSymbol$ is the pressure field, $\heatflux$ is the heat flux, and $\boldsymbol{\tau}$ is 
the viscous stress tensor defined as:
\begin{equation}
    \boldsymbol{\tau} = \mu (\gradientSymbol \velocity + \gradientSymbol \velocity^\top) 
    - \mu \frac{2}{3} \boldsymbol{I} (\gradientSymbol \cdot \velocity),
\end{equation}
with $\mu$ as viscosity. The airfoil surface is set as an adiabatic wall (zero heat flux) 
and the far field is set to the free stream conditions. 

The FOM is solved using a pseudo-time stepping method, specifically an implicit, local time stepping method, 
to march the solution forward from the initial condition (freestream fluid state) to the steady state solution.
The local time stepping technique allows for quicker convergence
by allowing the time step to vary between elements. This approach is typical for steady aerodynamics. We use converged steady solutions as simulation data to construct the reduced basis.

The full order model domain contains 8741 mesh points with a NACA0012 airfoil surface in the center. 
We use shape parameters for a parameterized ROM study. In this case, 
the airfoil shape is slightly altered using three Hicks-Henne bumps \cite{hhbumps} 
on the upper surface and three on the bottom surface. Hicks-Henne bumps are smooth bump formulations, 
and each shape parameter defines the amplitude of the Hicks-Henne bump.
The six-dimensional parameter space is sampled using 73 Latin hypercube samples, resulting in 73 total snapshots.
Both the FOM and ROM equations are solved using the open-source CFD code SU2 \cite{su2},
aided by the libROM code \cite{librom} from Lawrence Livermore National Lab for the snapshot collection and POD computation.

The laminar airfoil problem is a good academic problem for quick testing, because it does not 
require a large amount of computational resources. The 73 snapshots are collected on the Sherlock cluster
operated by the Stanford Research Computing Center. Then the reduced order model simulations are run serially 
on a macOS laptop with 2.4 GHz Intel Core i9 processor and 16 GB memory.

The POD basis has a dimension $\nbasisspace = 73$, but is then truncated to 20, 10, and 5 POD modes.
Figure~\ref{fig:airmodes} shows the first four modes, or singular vectors, for the density field from the POD computation. 
The variations of the flow field due to the airfoil shape differences are captured by the POD modes, so
the flow solution for any set of shape parameters can be approximated using a combination of these modes.
\begin{figure} [htb!]
    \centering
    \includegraphics[width=\linewidth]{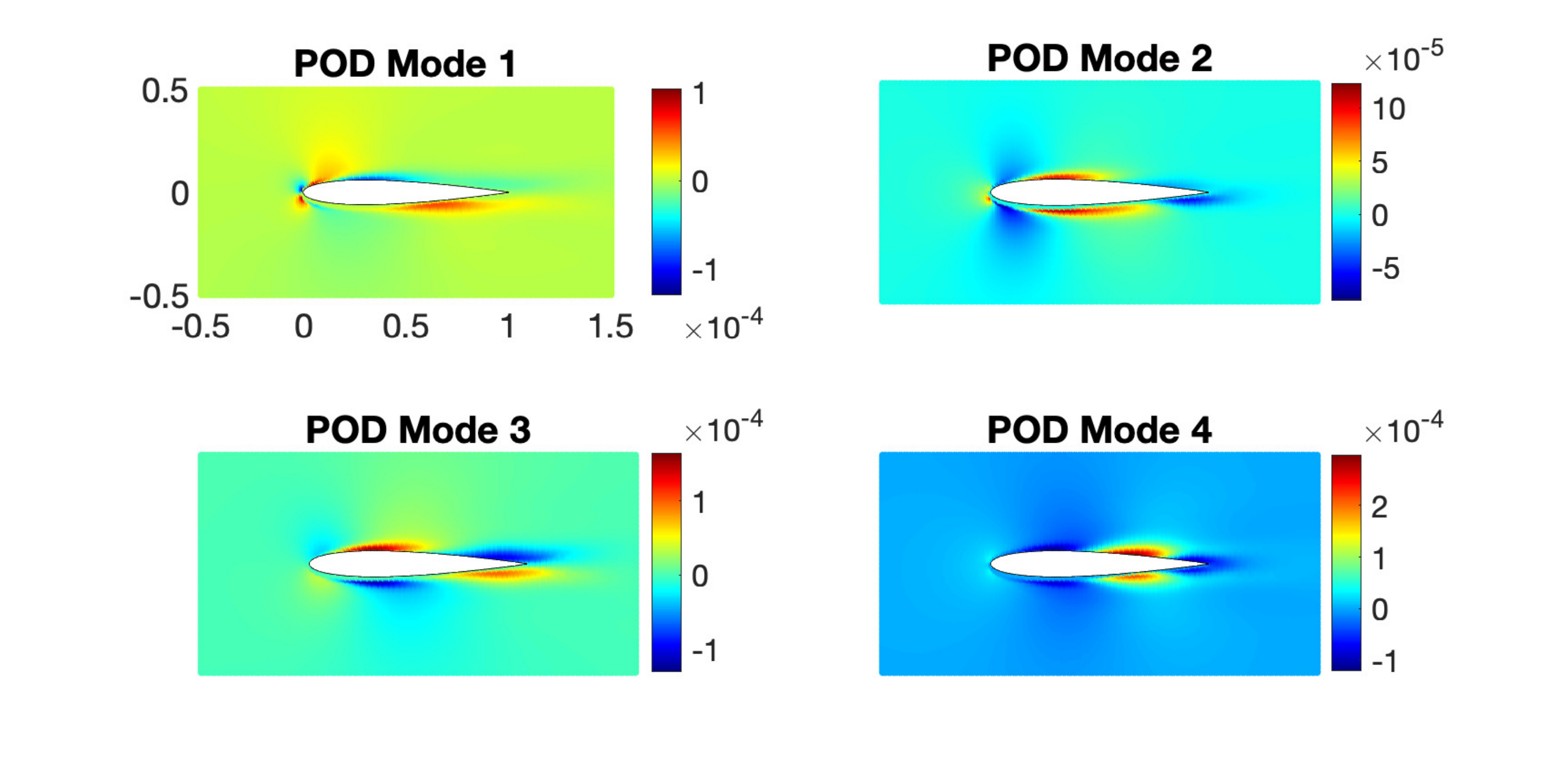}
    \caption{First four POD modes for the density variable for the laminar airfoil.}
    \label{fig:airmodes}
\end{figure}
Both sampling methods are tested and the results are shown in Figure~\ref{fig:air}. 
In all cases, the goal of the ROM is to predict the baseline NACA0012 airfoil shape, 
which is not included in the 73 snapshots.

It is also important to note that for finite-volume CFD models there are multiple equations 
at each node, so the input basis for the sampling algorithms needs to be condensed so that the size of the first 
dimension and the number of nodes are the same. We do this by taking the norm of the 
basis values at each node, resulting in one basis value per node. 
This is in contrast to Section~\ref{sec:laghydro}, where we also show a
different scenario of building a separate basis for a different field, and therefore, merging
the sampling indices from different fields. 

The measurements of error shown in Figure~\ref{fig:air} are defined as follows for the 
$L_2$ error and the maximum relative error:
\begin{align}\label{eq:errors-airfoil}
  \relError_{\sol, L_2} &= \frac{\| \sol - \sol_{\text{true}} \|_{L^2} }
                                {\| \sol_{\text{true}} \|_{L^2}},& 
  \relError_{\densitySymbol, \text{max}} &= \max\limits_{1 \leq i \leq \fullstaten} \left(
                                  \frac{|\densitySymbol_i - \densitySymbol_{i,\text{true}}|} 
                                       {|\densitySymbol_{i,\text{true}}|} \right) \times 100,
\end{align}
where the truth values are obtained from the full order model solution 
of the prediction case (the NACA0012 baseline shape) using SU2.

\begin{figure}[htb!]
     \centering
     \begin{subfigure}[b]{0.48\linewidth}
         \centering
         \includegraphics[width=\linewidth]{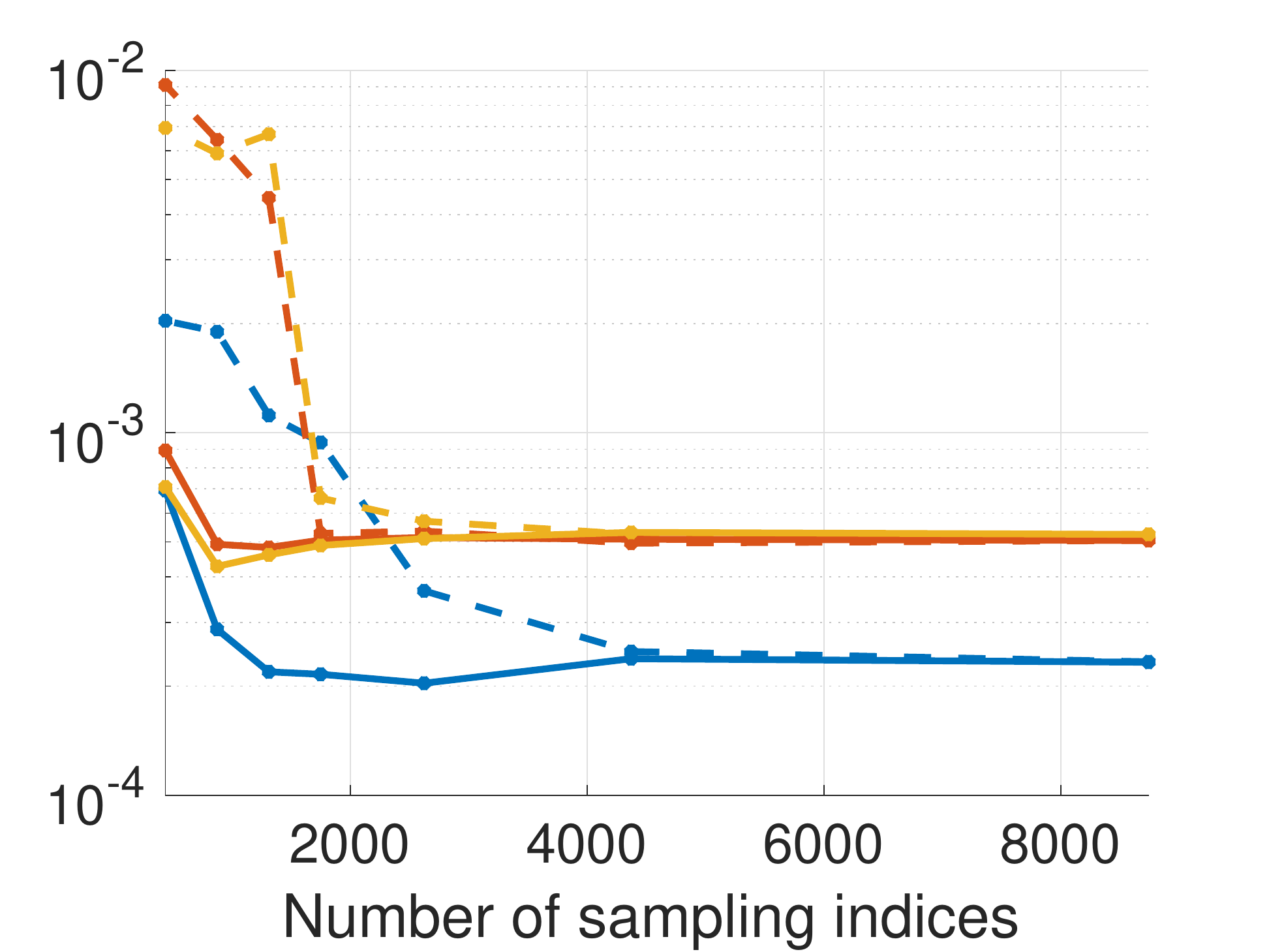}
         \caption{$L_2$ error of ROM prediction vs. true solution.}
     \end{subfigure}
     \hfill
     \begin{subfigure}[b]{0.48\linewidth}
         \centering
         \includegraphics[width=\linewidth]{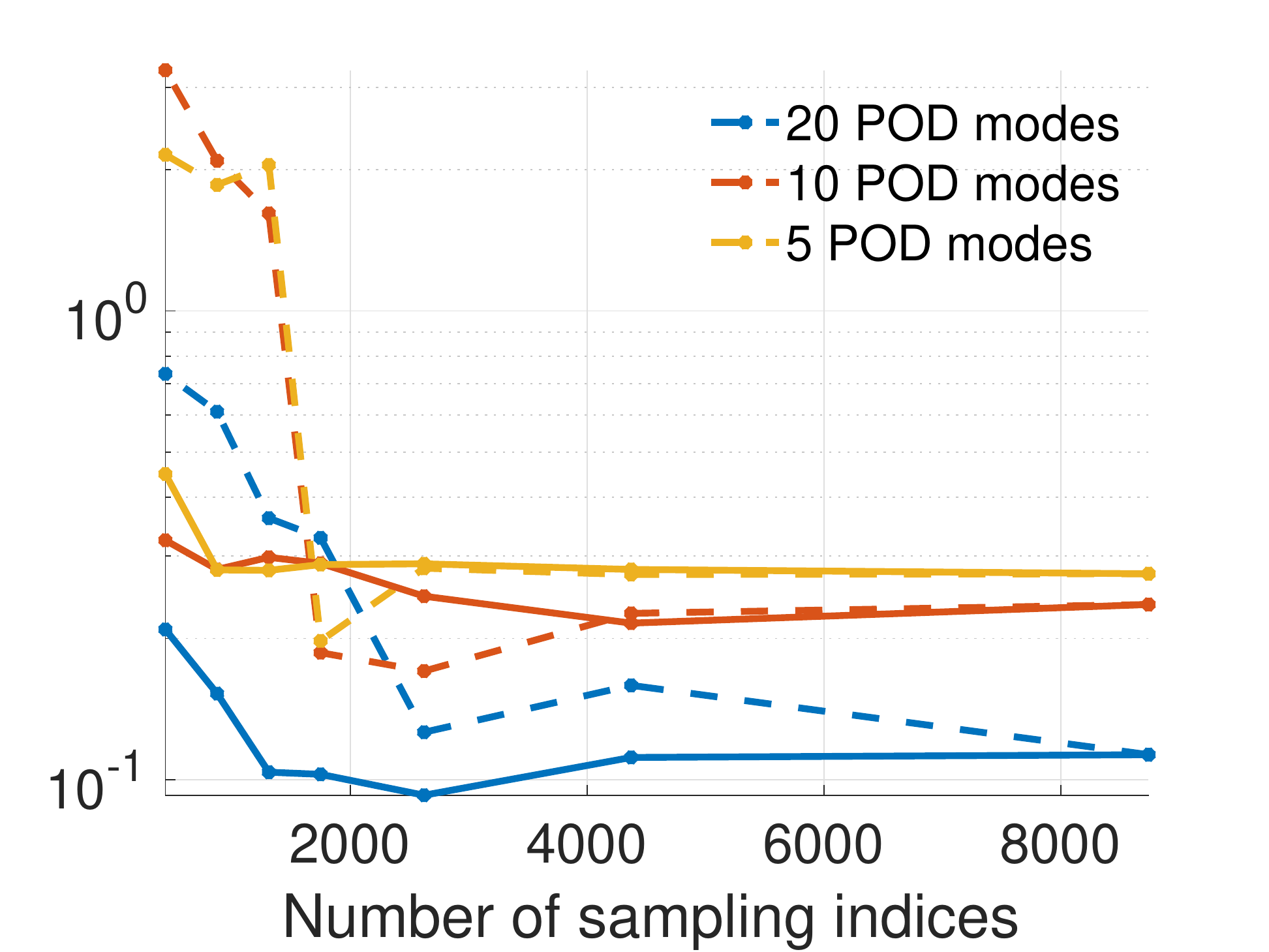}
         \caption{Maximum relative error percent for density.}
     \end{subfigure}
     \hfill
     \begin{subfigure}[b]{0.48\linewidth}
         \centering
         \includegraphics[width=\linewidth]{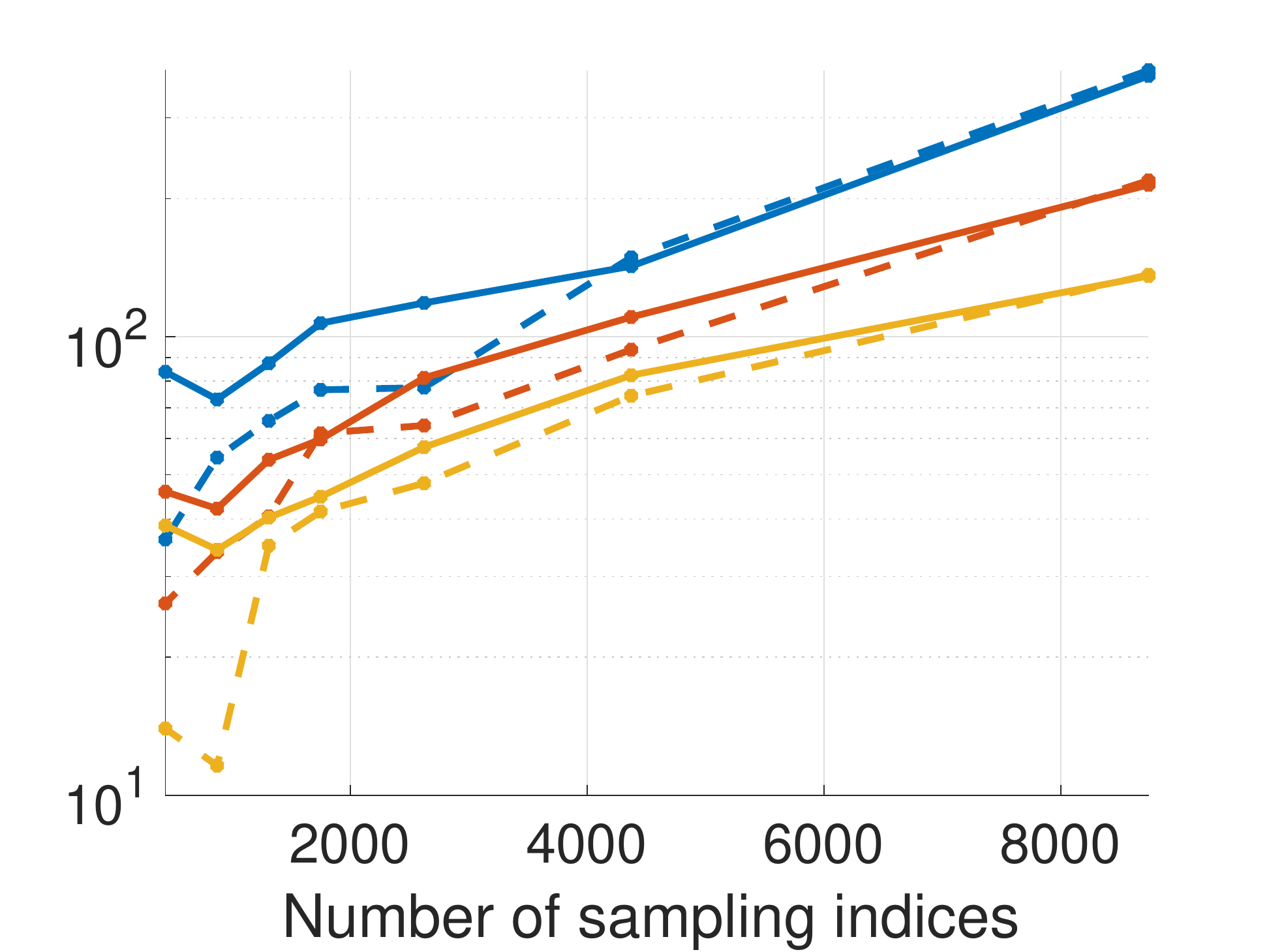}
         \caption{Simulation wall clock time in seconds. }
     \end{subfigure}
        \caption{Comparison of S-OPT (solid line) versus oversampled DEIM (dashed) algorithms. FOM size is 8741. }
        \label{fig:air}
\end{figure}
In Figure~\ref{fig:air} we see that as the number of sampled indices decreases, 
the S-OPT algorithm outperforms the DEIM algorithm as shown by the error measures. 
At the most extreme level of hyper-reduction, using only 437 samples, the ROM 
built with S-OPT algorithm and 5 POD modes performs better than the ROM built with DEIM and 20 POD modes.

However, there is a cost associated with the performance improvement. 
The ROMs built with the S-OPT algorithm generally take longer to converge. 
Figure~\ref{fig:mask} shows one reason the S-OPT algorithm ROMs take longer. 
Since this is a viscous problem, two levels of node-neighbors are required to compute 
the residual at each selected node. The S-OPT algorithm tends to choose nodes that are 
more spread out in the domain, while the DEIM algorithm chooses nodes near the 
airfoil surface where most of the changes from solution to solution take place. 
Since the set of neighboring nodes is greater for the S-OPT ROM, 
the total dimension for the hyper-reduced ROM is larger and each iteration takes longer. 
It is important to note that this dimension still scales with the reduced dimension, $\nbasisspace$.

\begin{figure}[htb!]
     \centering
     \begin{subfigure}[b]{0.45\textwidth}
         \centering
         \includegraphics[width=\textwidth]{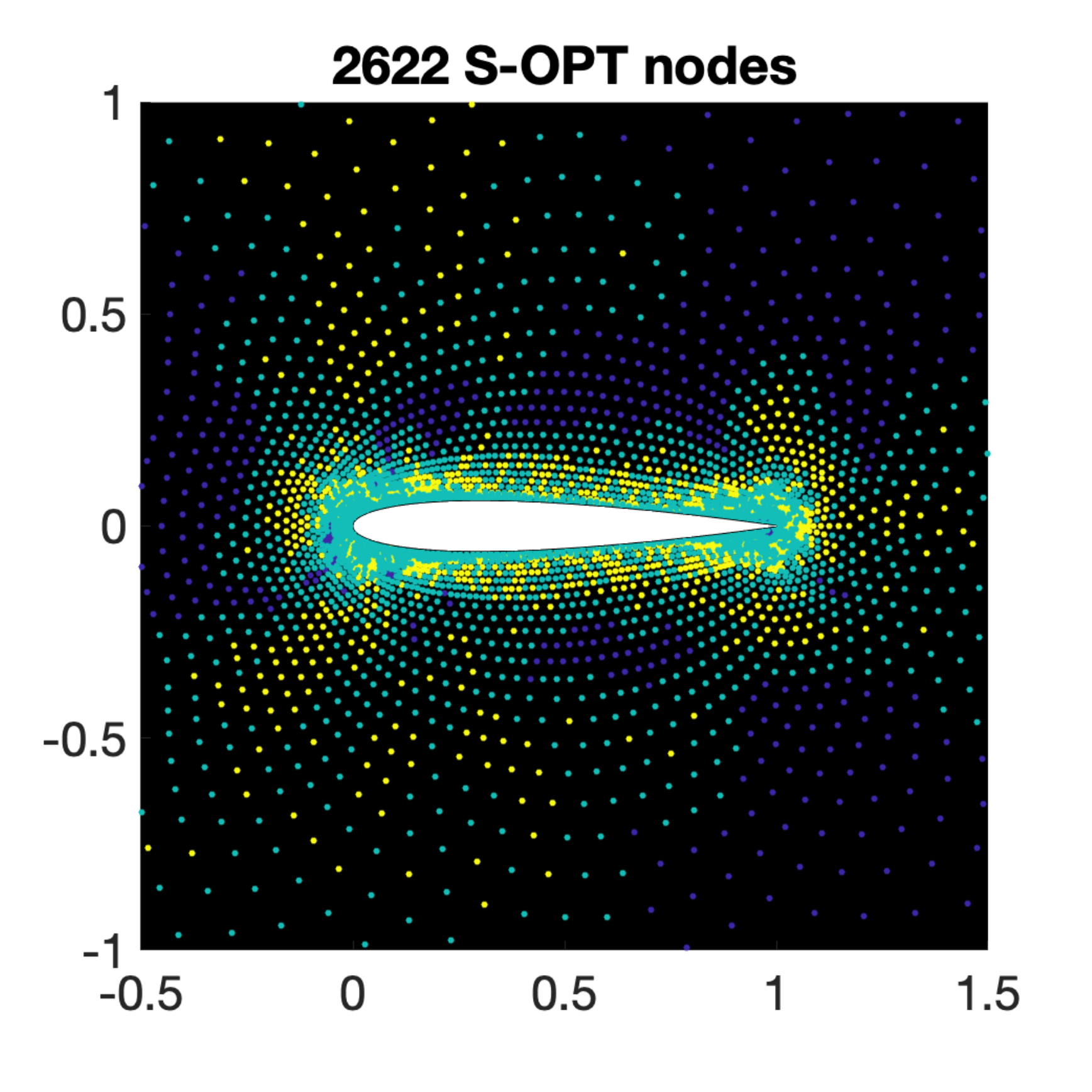}
     \end{subfigure}
     \hfill
     \begin{subfigure}[b]{0.45\textwidth}
         \centering
         \includegraphics[width=\textwidth]{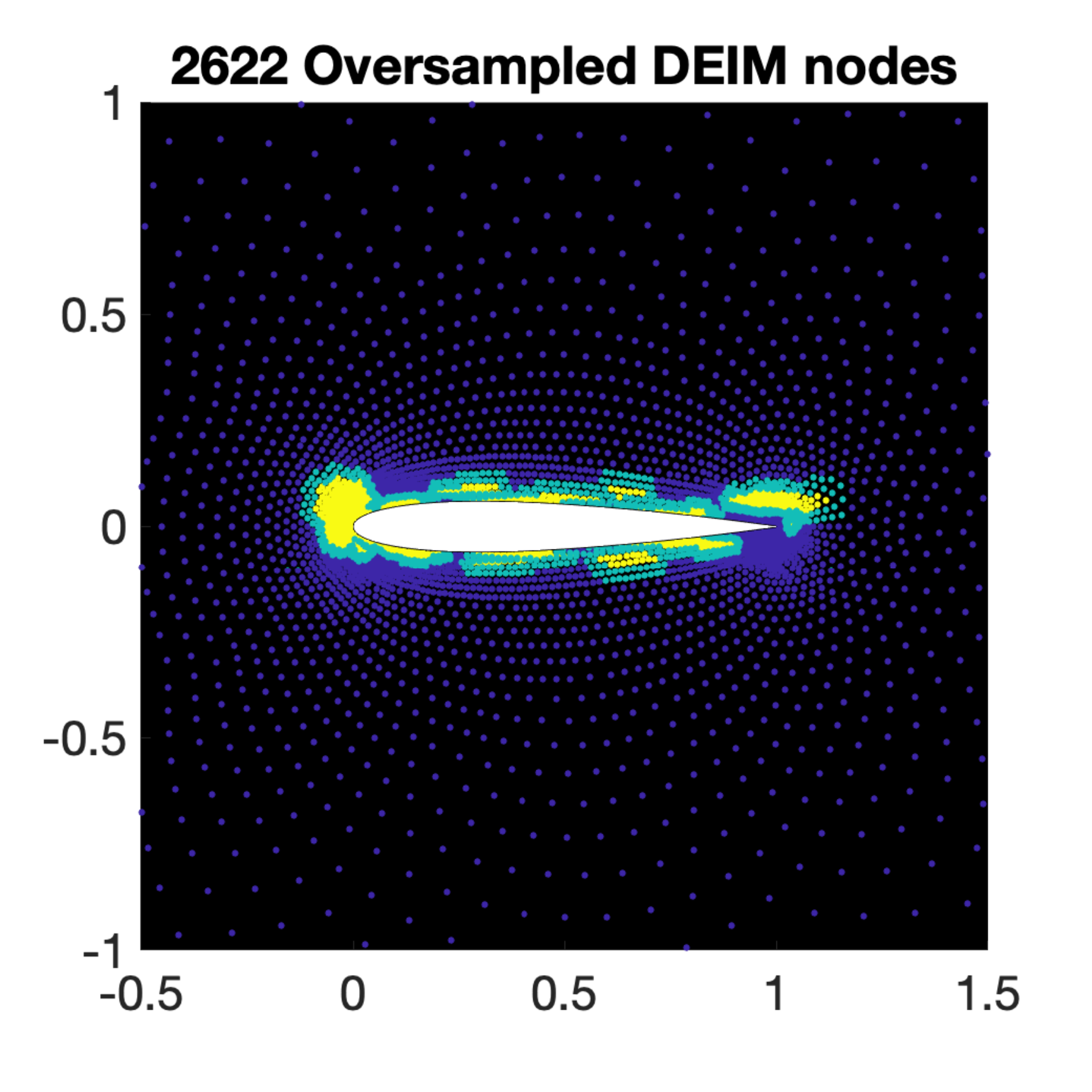}
     \end{subfigure}
        \caption{Partial domain showing node selection near the airfoil. Selected nodes are in yellow. Neighboring nodes required for residual computation are in cyan.}
        \label{fig:mask}
\end{figure}

\subsection{Lagrangian hydrodynamics}\label{sec:laghydro}
In the next two examples, we consider advection-dominated problems
arising in compressible gas dynamics. The Euler equation is used to model the
high-speed flow and shock wave propagation in a complex multimaterial setting,
and numerically solved in a moving Lagrangian frame \cite{harlow1971fluid}, 
assuming no external body force is exerted:
\begin{equation}\label{eq:euler}
  \begin{aligned}
    \text{momentum conservation}:& & \densitySymbol\frac{d\velocitySymbol}{d\timeSymbol} &=
    \gradientSymbol \cdot \stressSymbol \\
    \text{mass conservation}:& & \dfrac{1}{\densitySymbol}\frac{d\densitySymbol}{d\timeSymbol} &=
    -\gradientSymbol \cdot \velocitySymbol \\
    \text{energy conservation}:& & \densitySymbol\frac{d\energySymbol}{d\timeSymbol} &=
    \stressSymbol :  \gradientSymbol \velocitySymbol \\
    \text{equation of motion}:& & \frac{d\positionSymbol}{d\timeSymbol} &= \velocitySymbol.
  \end{aligned}
\end{equation}
Here, $\frac{d}{d\timeSymbol} = \frac{\partial}{\partial \timeSymbol} + 
\velocitySymbol \cdot \gradientSymbol$ is the material derivative, 
$\densitySymbol$ denotes the density of the fluid, $\positionSymbol$ and
$\velocitySymbol$ denote the position and the velocity of the particles in a
deformable medium $\solDomainSymbol(\timeSymbol)$ in the Eulerian coordinates,
$\stressSymbol$ denotes the deformation stress tensor, and $\energySymbol$
denotes the internal energy per unit mass. In gas dynamics,
the stress tensor is isotropic, and we write $\stressSymbol = -\pressureSymbol
\identitySymbol + \artificialStressSymbol$, where $\pressureSymbol$ denotes the 
thermodynamic pressure, and $\artificialStressSymbol$ denotes the artificial
viscosity stress. The thermodynamic pressure is described by the equation of
state, and can be expressed as a function of the density and the internal
energy. With the assumption of polytropic ideal gas with an 
adiabatic index $\adiabaticIndexSymbol > 1$, which yields the equation of state 
\begin{equation}\label{eq:EOS}
  \pressureSymbol = (\adiabaticIndexSymbol - 1) \densitySymbol \energySymbol.
\end{equation}
The system is prescribed with an initial condition and a boundary condition
$\velocitySymbol \cdot \normalSymbol = \neumannSymbol$, where $\normalSymbol$ is
the outward normal unit vector on the domain boundary. 

Using a high-order curvilinear finite element method (FEM) 
for Lagrangian hydrodynamics \cite{dobrev2012high} 
for semi-discretization of the Euler equation
results in the differential system:  
\begin{equation}\label{eq:laghos}
  \begin{aligned}
    \text{momentum conservation}:& & \kinematicMassMat\frac{d\velocity}{d\timeSymbol} &=
    -\forceMat(\velocity, \energy, \position; \param) \cdot\oneVec \\
    \text{energy conservation}:& & \thermodynamicMassMat\frac{d\energy}{d\timeSymbol} &=
    \forceMat(\velocity, \energy, \position; \param)^\top\cdot\velocity \\
    \text{equation of motion}:& & \frac{d\position}{d\timeSymbol} &= \velocity,
  \end{aligned}
\end{equation}
where $\velocity, \energy, \position: [0,T] \to \mathbb{R}$ denotes the 
FEM coefficient vector functions for velocity $\velocitySymbol$, 
internal energy density $\energySymbol$ and position $\positionSymbol$.
In order to obtain a fully discretized system of equations, we apply 
RK2-average scheme as the time integrator, 
a modification of the midpoint Runge--Kutta second-order scheme. 
The RK2-average scheme is written as
\begin{align}\label{eq:RK2-avg}
  \velocityt{\timeIndex+\frac{1}{2}} &= \velocityt{\timeIndex} - (\timestepk{\timeIndex}/2)
    \kinematicMassMat^{-1} \forceOnek{\timeIndex}, &
    \velocityt{\timeIndex+1} &= \velocityt{\timeIndex} - \timestepk{\timeIndex}
    \kinematicMassMat^{-1} \forceOnek{\timeIndex+\frac{1}{2}}, \\
  \energyt{\timeIndex+\frac{1}{2}} &= \energyt{\timeIndex} + (\timestepk{\timeIndex}/2)
    \thermodynamicMassMat^{-1} \forceTvk{\timeIndex}, &
    \energyt{\timeIndex+1} &= \energyt{\timeIndex} + \timestepk{\timeIndex}
    \thermodynamicMassMat^{-1} \avgforceTvk{\timeIndex+\frac{1}{2}}, \\
  \positiont{\timeIndex+\frac{1}{2}} &= \positiont{\timeIndex} + (\timestepk{\timeIndex}/2)
    \velocityt{\timeIndex+\frac{1}{2}}, & \positiont{\timeIndex+1} &=
    \positiont{\timeIndex} + \timestepk{\timeIndex} \avgvelocityt{\timeIndex+\frac{1}{2}},
\end{align}
where the state $\statet{\timeIndex} 
= (\velocityt{\timeIndex}; \energyt{\timeIndex}; 
\positiont{\timeIndex})^\top \in\RR^{\sizeWholeFE}$ 
is used to compute the updates 
\begin{align}
\forceOnek{\timeIndex} & = \left(\forceMat (\statet{\timeIndex}) \right ) \cdot\oneVec, &
\forceTvk{\timeIndex} & = \left(\forceMat (\statet{\timeIndex}) \right )^\top \cdot 
\velocityt{\timeIndex+\frac{1}{2}},
\end{align}
in the first stage. Similarly, $\statet{\timeIndex+\frac{1}{2}} = 
(\velocityt{\timeIndex+\frac{1}{2}}; \energyt{\timeIndex+\frac{1}{2}}; 
\positiont{\timeIndex+\frac{1}{2}})^\top \in\RR^{\sizeWholeFE}$ 
is used to compute the updates 
\begin{align}
\forceOnek{\timeIndex+\frac{1}{2}} 
& = \left(\forceMat (\statet{\timeIndex+\frac{1}{2}}) \right ) \cdot \oneVec, &
\avgforceTvk{\timeIndex+\frac{1}{2}} & = \left
(\forceMat (\statet{\timeIndex + \frac{1}{2}}) \right )^\top \cdot
\avgvelocityt{\timeIndex+\frac{1}{2}},
\end{align}
with $\avgvelocityt{\timeIndex+\frac{1}{2}} = (\velocityt{\timeIndex} +
\velocityt{\timeIndex+1})/2$ in the second stage.  
Since explicit Runge-Kutta methods are used, we need to control the time step
size in order to maintain the stability of the fully discrete schemes.  We
follow the automatic time step control algorithm described in Section 7.3 of
\cite{dobrev2012high}. 

We use a linear-subspace reduced order model technique for \eqref{eq:laghos} as developed in \cite{copeland2022reduced}. The solution nonlinear subspace (SNS) procedure in \cite{choi2020sns} is used
to build the nonlinear term bases $\forceOneBasis$ and $\forceTvBasis$
on the right-hand-side of the momentum conservation equation and energy conservation equation separately, and the sampling indices for hyper-reduction of each of nonlinear terms are used to construct the sampling matrices $\forceOneSamplingMat$ and $\forceTvSamplingMat$, respectively. 
In the following benchmark experiments, we compare the numerical results of reduced order model using the sampling algorithms discussed in Section~\ref{sec:salgs}, i.e., the oversampled DEIM and S-OPT.
With the automatic time step control algorithm, 
it is very likely that the temporal
discretization used in the hyper-reduced system is different from the full
order model even with the same problem setting.  To this end, we denote by
$\ntimestepROM$ the number of time steps in the fully discrete hyper-reduced
system, to differentiate it from the notation $\ntimestep$ for the full order
model. The $L^2$ relative error for each ROM field is measured
against the corresponding FOM solution at the final time $\finalTime$, which is
defined as:
\begin{align}\label{eq:relerrors-laghos}
  \relerrorVelocityt{L^2} &= \frac{\| \velocityt{\ntimestep} -
  \velocityApproxt{\ntimestepROM} \|_{L^2} }{\| \velocityt{\ntimestep} \|_{L^2}},& 
  \relerrorEnergyt{L^2} &= \frac{\| \energyt{\ntimestep} - 
  \energyApproxt{\ntimestepROM} \|_{L^2} }{\| \energyt{\ntimestep} \|_{L^2}},& 
  \relerrorPositiont{L^2} &= \frac{\| \positiont{\ntimestep} -
  \positionApproxt{\ntimestepROM} \|_{L^2} }{\| \positiont{\ntimestep} \|_{L^2}}.
\end{align}
All the Lagrangian hydrodynamics simulations in this
subsection use Quartz in Livermore Computing Center\footnote{High
performance computing at LLNL, https://hpc.llnl.gov/hardware/platforms/Quartz},
on Intel Xeon CPUs with 128 GB memory, peak TFLOPS of 3251.4, and peak single
CPU memory bandwidth of 77 GB/s.

\subsubsection{2D Gresho vortex}\label{sec:gresho}
The Gresho vortex problem is a two-dimensional standard benchmark test 
for the incompressible inviscid Navier--Stokes equations \cite{gresho1990theory}.
In this problem, we consider a manufactured smooth solution from extending the 
steady state Gresho vortex solution to the compressible Euler equations. 
For the detailed description of the set-up of the Gresho vortex problem, 
we refer the readers to Sections 6.1.1 and 6.2 of \cite{copeland2022reduced}. 
The final time is taken as $T=0.1$. 
Figure~\ref{fig:gresho_sol} shows the initial condition and the final-time solution 
for 2D Gresho vortex. It can be seen that the vortex is rotating. 

\begin{figure}[htb!]
    \centering
    \includegraphics[width=0.48\linewidth]{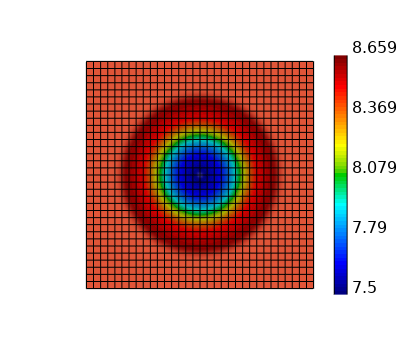}
    \hfill
    \includegraphics[width=0.48\linewidth]{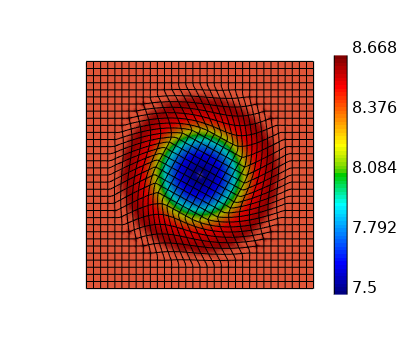}
    \caption{Initial condition (left) and final-time solution (right) for 2D Gresho vortex.}
    \label{fig:gresho_sol}
\end{figure}

We first investigate how the sampling algorithms affect the projection error of some sampled snapshots of the right hand side of the energy conservation equation. Following the notations in  Theorem~\ref{thm:error-estimate}, with 
$\bm{b} = \forceTvk{\timeIndex}$, $\bm{M} = \forceTvBasis$, and  
$\samplingmat = \forceTvSamplingMat$, we examine the oblique projection error 
$\|\forceTvk{\timeIndex} - \tilde{\forceTvk{\timeIndex}}(\forceTvSamplingMat)\|$ 
and use the orthogonal projection error 
$\|\text{proj}_{\forceTvBasis}^\perp \forceTvk{\timeIndex} \|$ as reference.
In Figure~\ref{fig:gresho_snap}, we illustrate the effects of the choice of sampling algorithm 
with varying number of sampling indices on the oblique projection error in some nonlinear snapshot samples, which is a crucial component of error bounding in PROM for nonlinear problems.

In this test, the dimension of the nonlinear term subspaces, i.e. the number of columns
in the nonlinear term bases, of the nonlinear term in momentum conservation equation and energy conservation equation is 32 and 72, respectively.
In each of the nonlinear term evaluations, the number of sampling indices is taken as the 
product of the nonlinear term basis dimension and the oversampling ratio, 
which takes a value between 2 and 15. 
The oblique projection error in the S-OPT
has a much faster and smoother decay to the orthogonal projection error 
than the oversampled DEIM selection for all the selected snapshot samples.

In Figure~\ref{fig:gresho_error}, we depict the final-time $L^2$ error of the reduced order model
solution against the number of sampling indices for both sampling algorithms. 
When the oversampled DEIM algorithm with the oversampling ratio less than 5 is used, 
the reduced order model is unstable and is not able to yield meaningful approximation. 
Moreover, with the oversampling ratio between 5 and 11, the final-time error for DEIM is oscillatory 
and significantly larger than the S-OPT
for all solution components in Lagrangian hydrodynamics.  
This suggests the advantage of both accuracy and stability by 
using the S-OPT over the oversampled DEIM selection. 
Figure~\ref{fig:gresho_spmesh} depicts the sample mesh, which consists of 
all the elements containing a sampling node, of both sampling algorithms 
at an oversampling ratio of 3. Similar to our previous observation, 
S-OPT tends to sample nodes in a more widespread manner. 
In this example, DEIM densely selects nodes close to the rotating vortex, 
as shown in Figure~\ref{fig:gresho_sol}. 

\begin{figure}[htb!]
    \centering
    \includegraphics[width=0.32\linewidth]{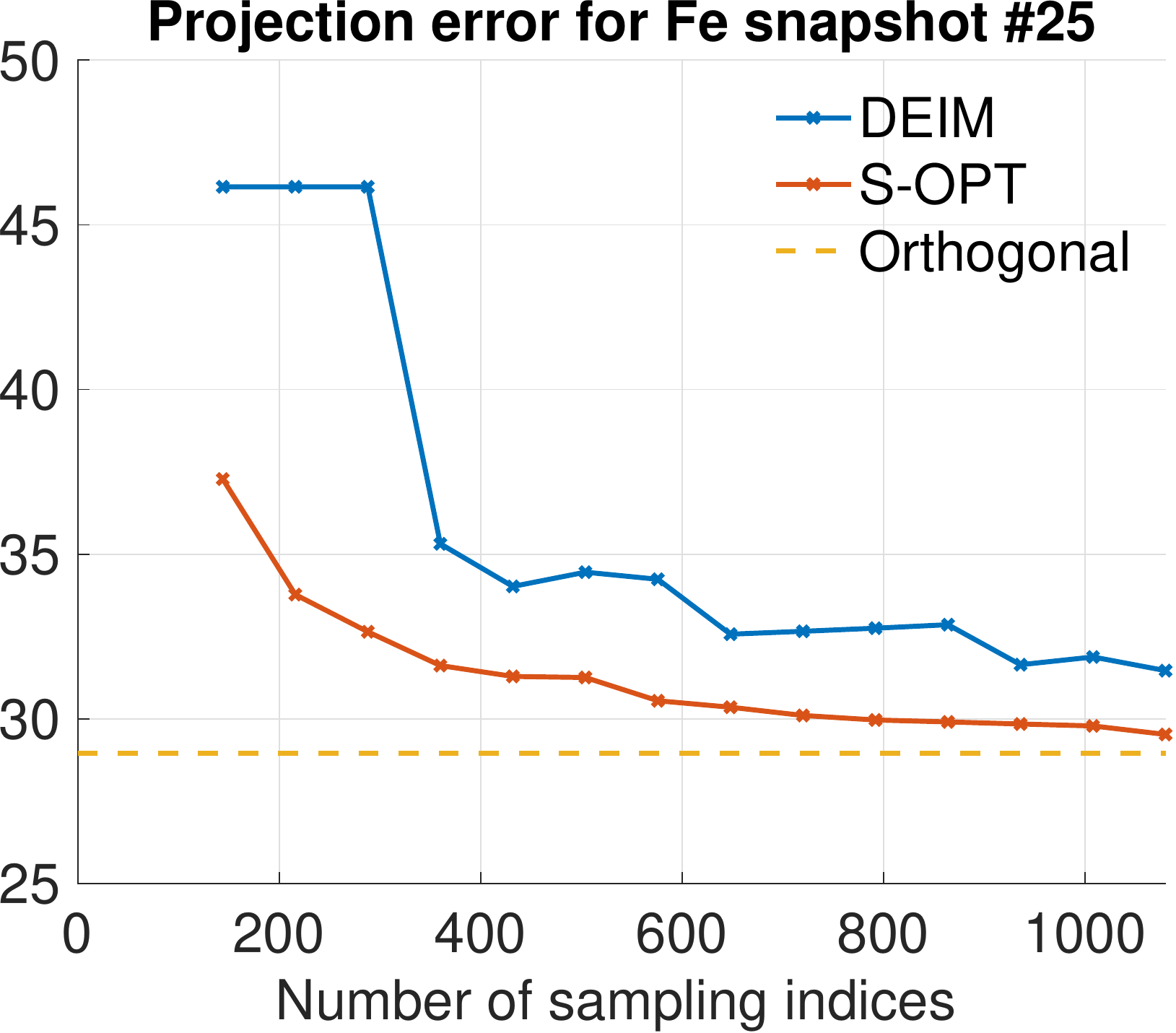}
    \hfill
    \includegraphics[width=0.32\linewidth]{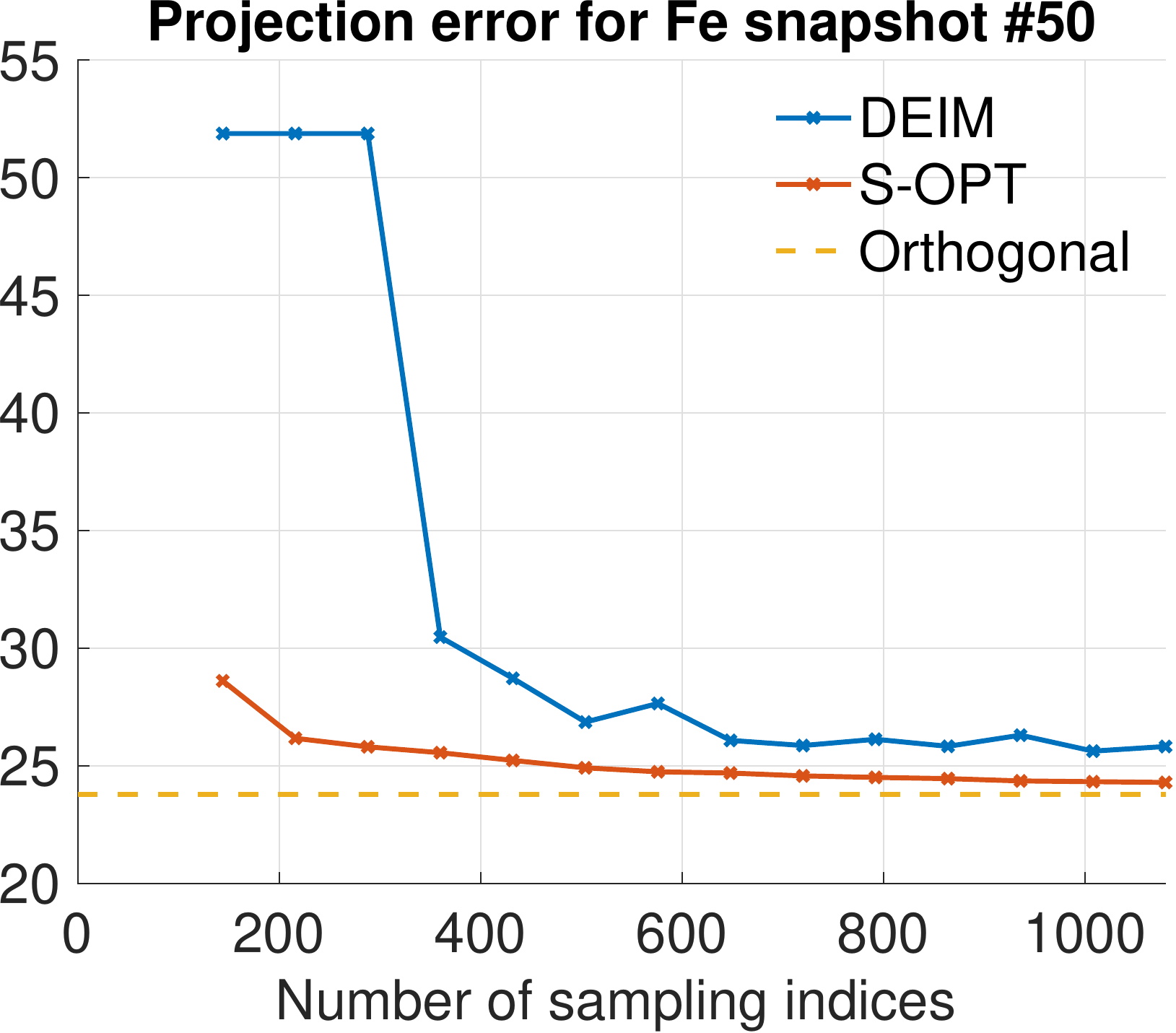}
    \hfill
    \includegraphics[width=0.32\linewidth]{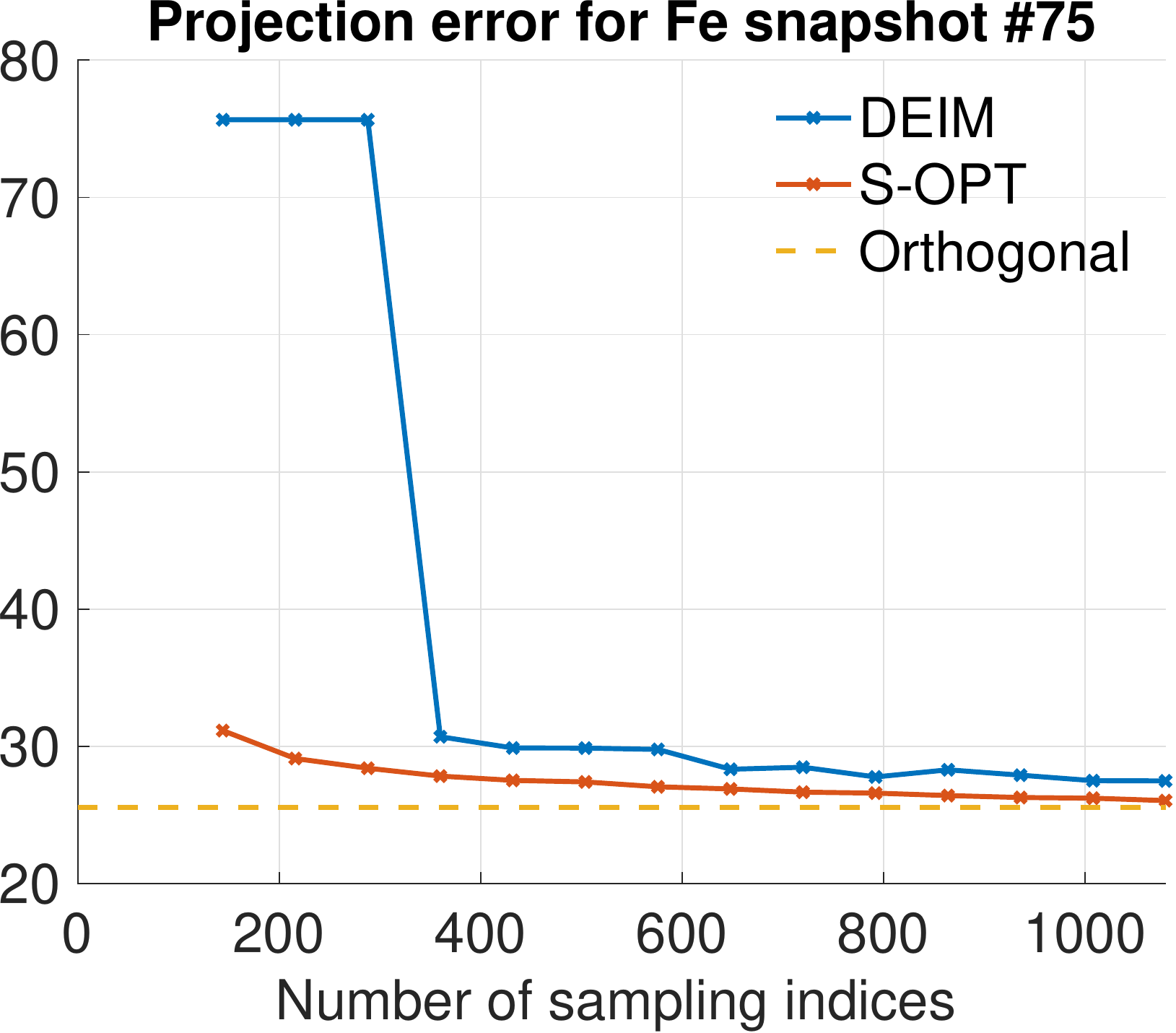}
    \caption{Oblique projection error in some snapshot samples of energy nonlinear term with varying number of sampling indices in Gresho vortex problem.}
    \label{fig:gresho_snap}
\end{figure}

\begin{figure}[htb!]
    \centering
    \includegraphics[width=0.32\linewidth]{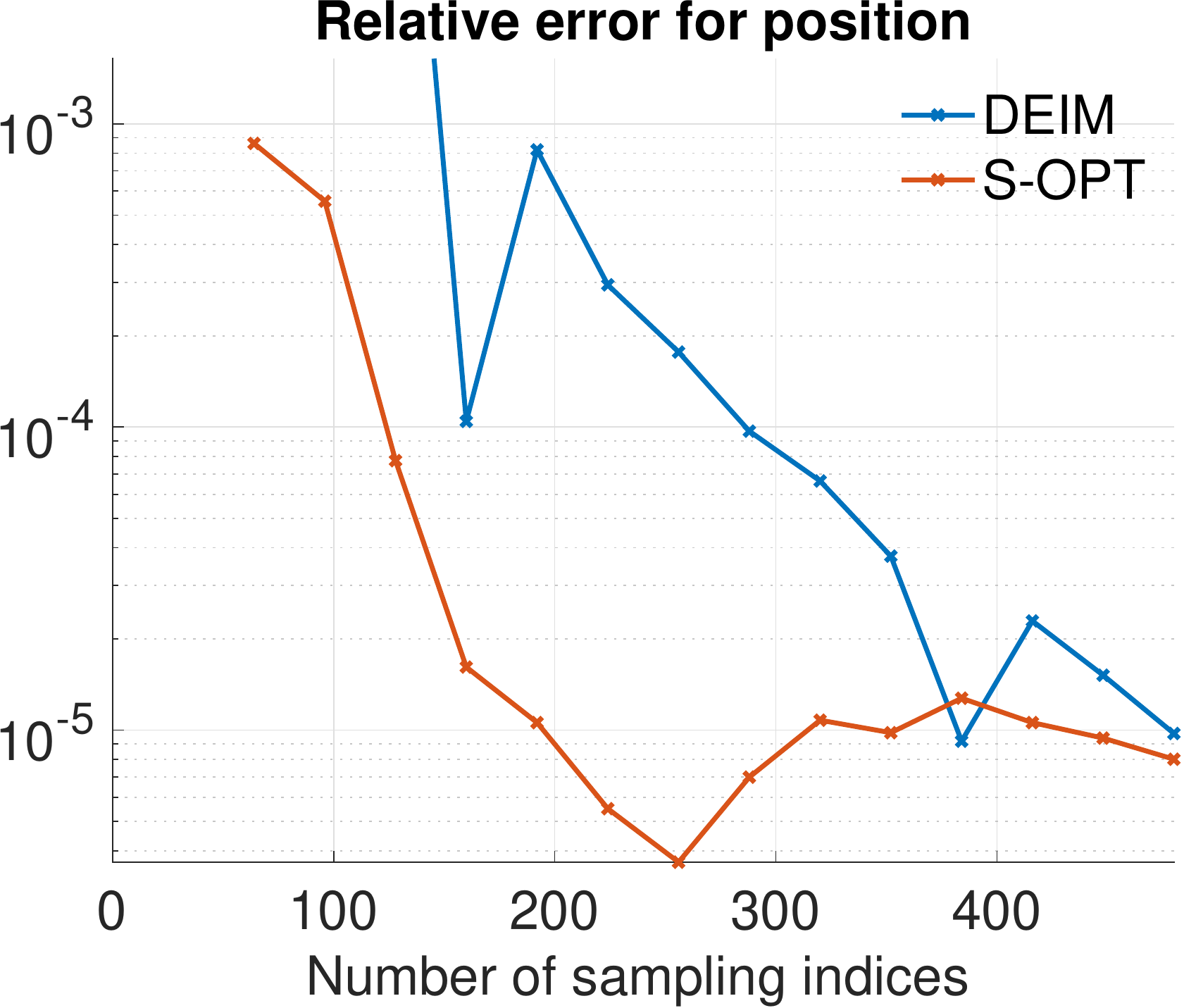}
    \hfill
    \includegraphics[width=0.32\linewidth]{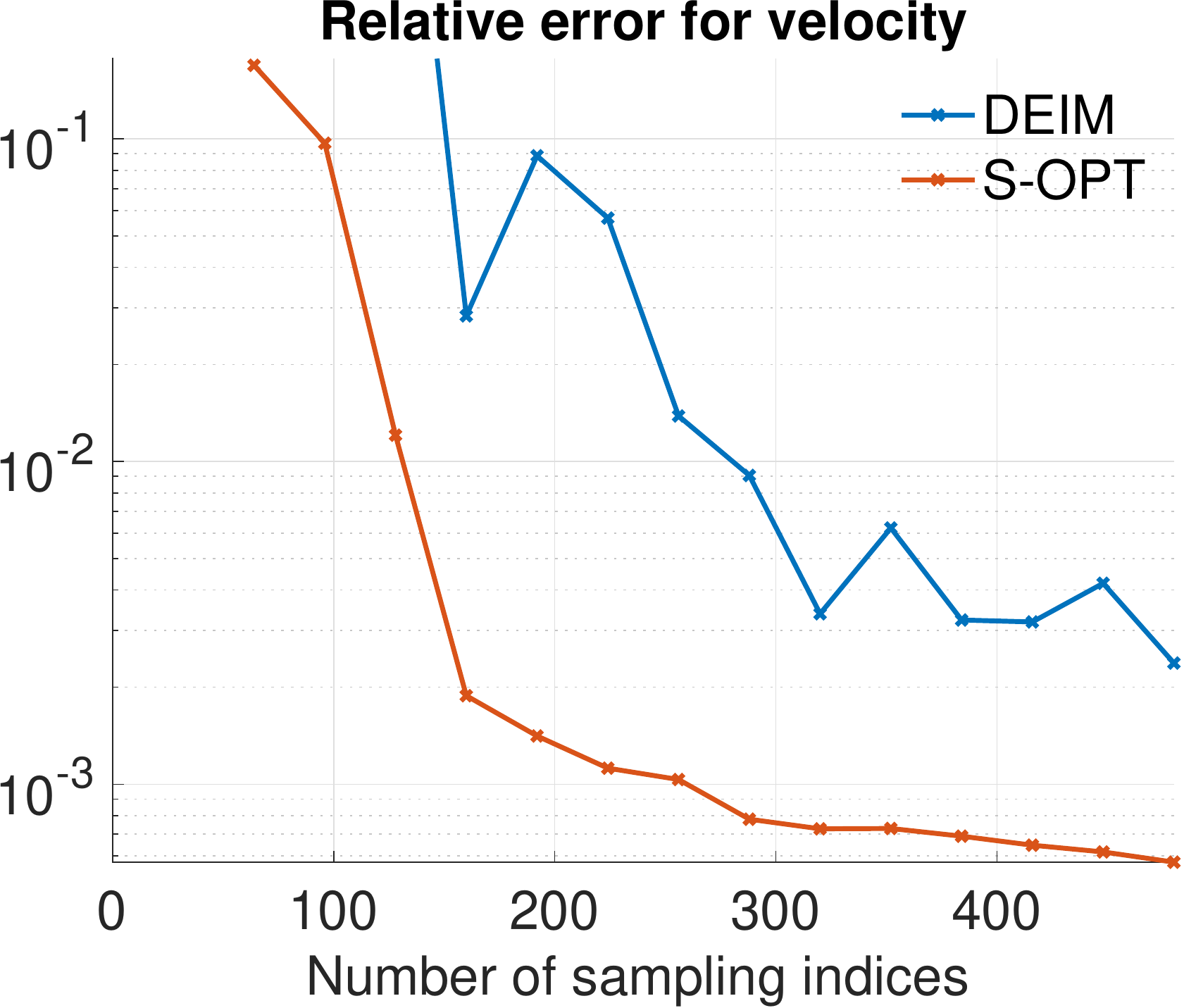}
    \hfill
    \includegraphics[width=0.32\linewidth]{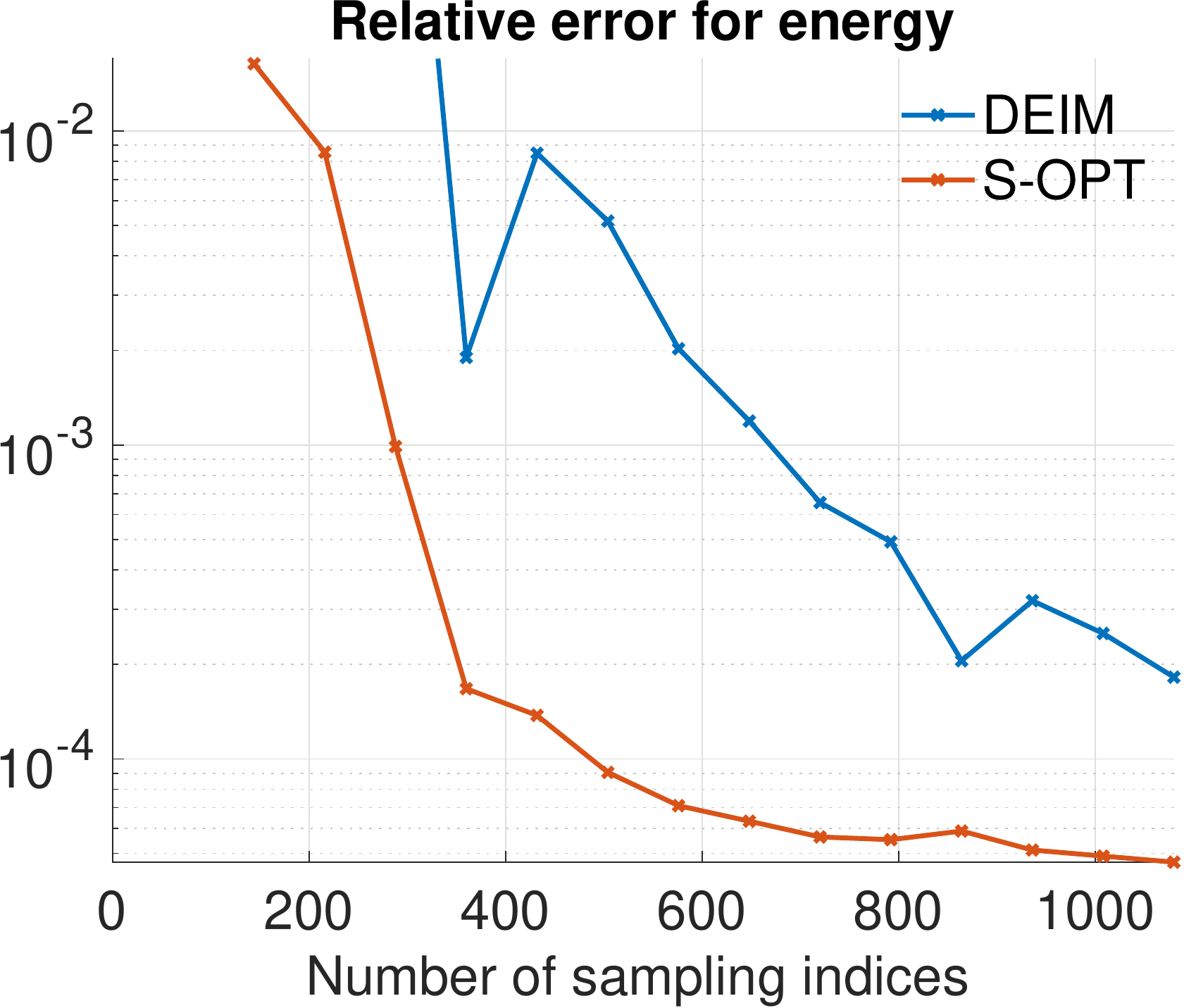}
    \caption{Final time solution error with varying number of sampling indices in Gresho vortex problem.}
    \label{fig:gresho_error}
\end{figure}

\begin{figure}[htb!]
     \centering
     \begin{subfigure}[b]{0.45\textwidth}
         \centering
         \includegraphics[width=\textwidth]{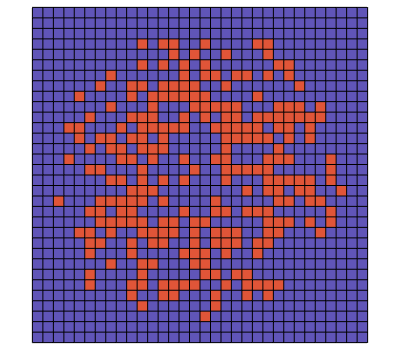}
     \end{subfigure}
     \hfill
     \begin{subfigure}[b]{0.45\textwidth}
         \centering
         \includegraphics[width=\textwidth]{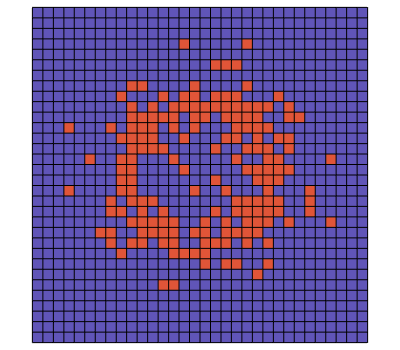}
     \end{subfigure}
        \caption{Sample mesh for Gresho vortex of S-OPT (left) and DEIM (right) sampling algorithms.}
        \label{fig:gresho_spmesh}
\end{figure}

\subsubsection{3D Sedov blast}\label{sec:sedov}
The Sedov blast problem is a three-dimensional standard shock hydrodynamic benchmark test
\cite{sedov1993similarity}, where we consider a delta source of
internal energy initially deposited at the origin of a three-dimensional cube. 
The final time is taken as $T=0.1$. 
For the detailed description of the set-up of the Sedov Blast problem, 
we refer the readers to Sections 6.1.2 and 6.2 of \cite{copeland2022reduced}. 

\begin{figure}[htb!]
    \centering
    \includegraphics[width=0.48\linewidth]{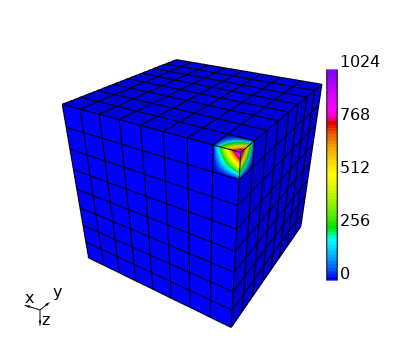}
    \hfill
    \includegraphics[width=0.48\linewidth]{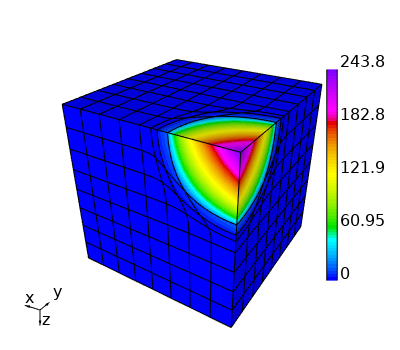}
    \caption{Initial condition (left) and final-time solution (right) for 3D Sedov blast.}
    \label{fig:sedov_sol}
\end{figure}

Again, we begin with investigating how the sampling algorithms affect the projection error of some sampled snapshots of the right hand side of the energy conservation equation by examining the oblique projection error 
$\|\forceTvk{\timeIndex} - \tilde{\forceTvk{\timeIndex}}(\forceTvSamplingMat)\|$, with 
the orthogonal projection error 
$\|\text{proj}_{\forceTvBasis}^\perp \forceTvk{\timeIndex} \|$ as reference.
In Figure~\ref{fig:sedov_snap}, we illustrate the effects of the choice of sampling algorithm 
with varying number of sampling indices on the oblique projection error in some nonlinear snapshot samples, 
which, again, is a crucial component of error bound in PROM for nonlinear problems. 
In this test, the dimensions of the nonlinear term subspaces, i.e. the number of columns
in the nonlinear term bases, for the momentum conservation equation and energy conservation equation are 53 and 13, respectively.
In each of the nonlinear term evaluations, the number of sampling indices is taken as the 
product of the nonlinear term basis dimension and the oversampling ratio, 
which takes value between 2 and 15. 
Similar to the Gresho vortex problem, 
while the oblique projection error in both sampling algorithms 
asymptotically decays to the orthogonal projection error, 
it can be observed that the decay is much faster  
with the S-OPT selection than the oversampled DEIM selection 
for all the selected snapshot samples. 
In Figure~\ref{fig:sedov_error}, we depict the final-time $L^2$ error of the 
solution against the number of sampling indices for both sampling algorithms. 
Again, it can be seen that for all solution components in Lagrangian hydrodynamics, 
the S-OPT selection gives much more stable decay 
in the solution error than the oversampled DEIM selection. 

\begin{figure}[htb!]
    \centering
    \includegraphics[width=0.32\linewidth]{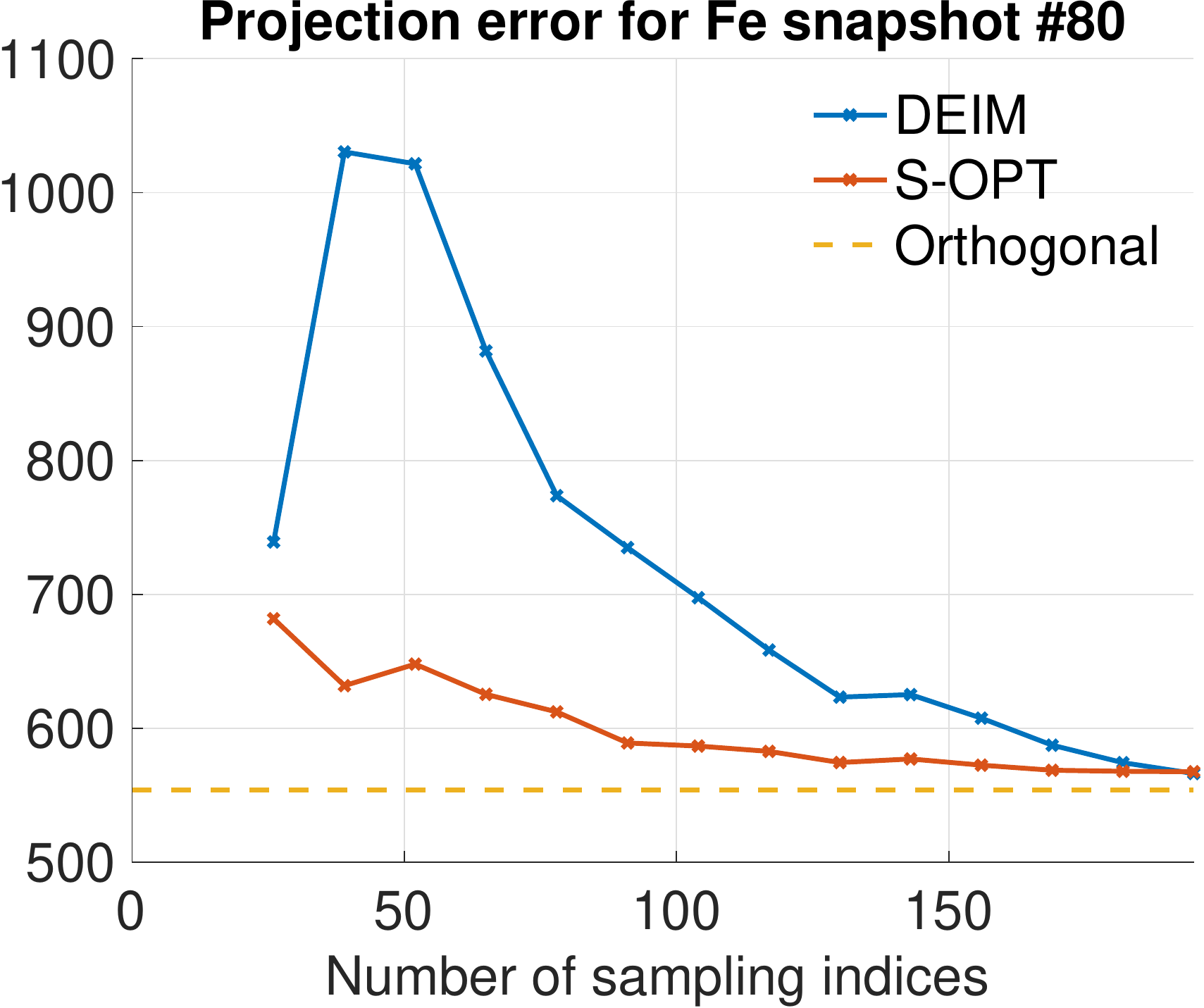}
    \hfill
    \includegraphics[width=0.32\linewidth]{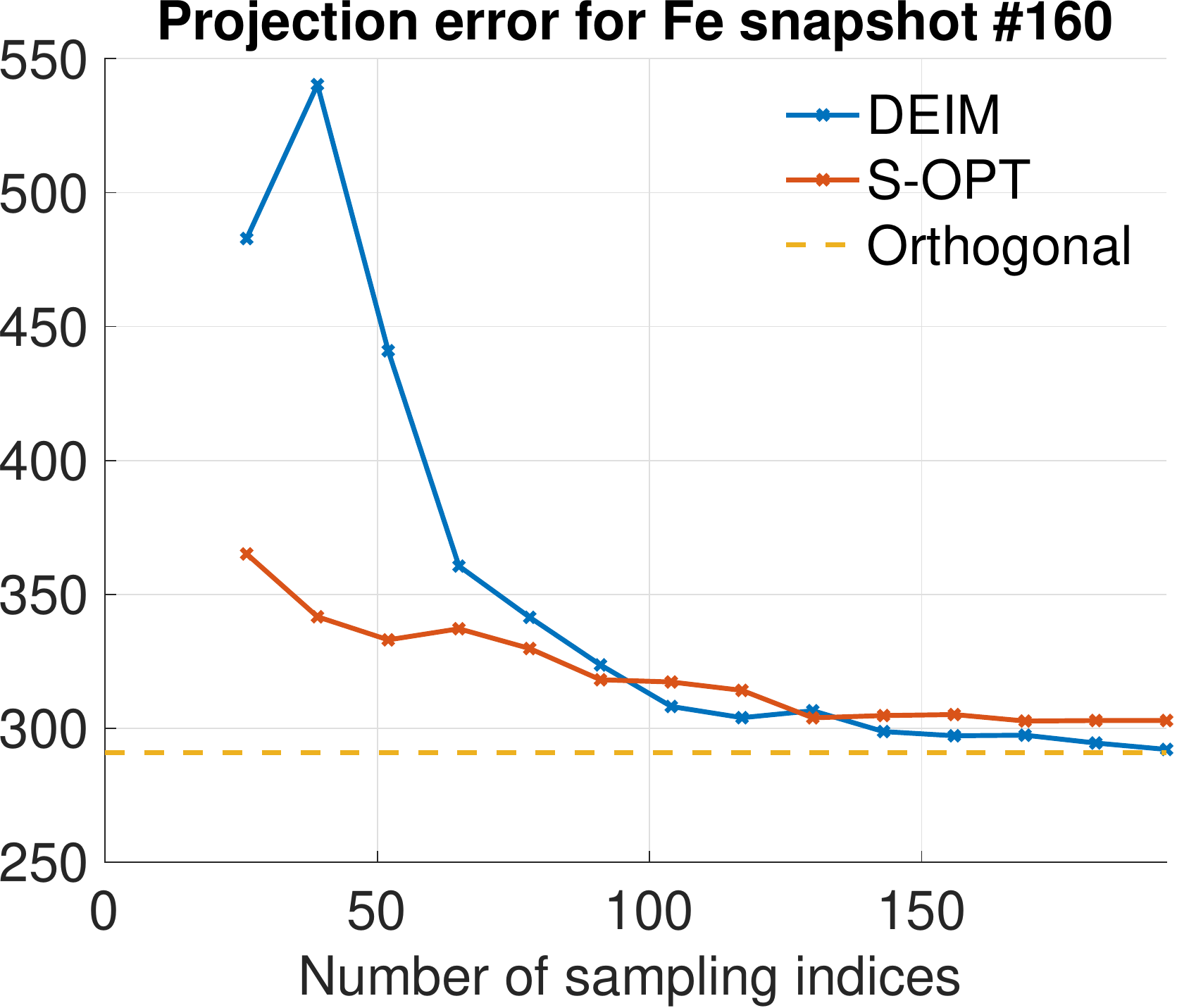}
    \hfill
    \includegraphics[width=0.32\linewidth]{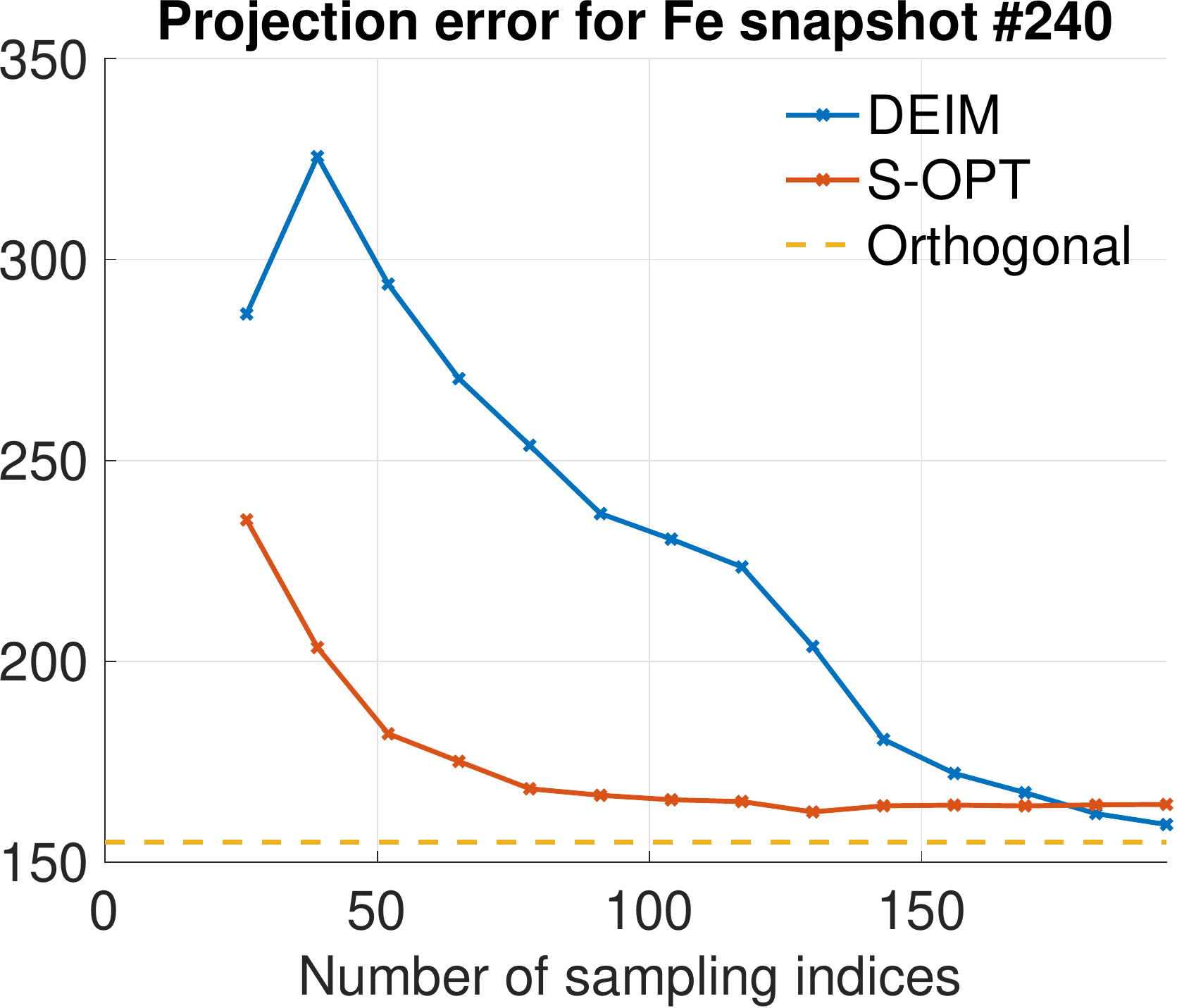}
    \caption{Oblique projection error in some snapshot samples of energy nonlinear term with varying number of sampling indices in Sedov blast problem.}
    \label{fig:sedov_snap}
\end{figure}

\begin{figure}[htb!]
    \centering
    \includegraphics[width=0.32\linewidth]{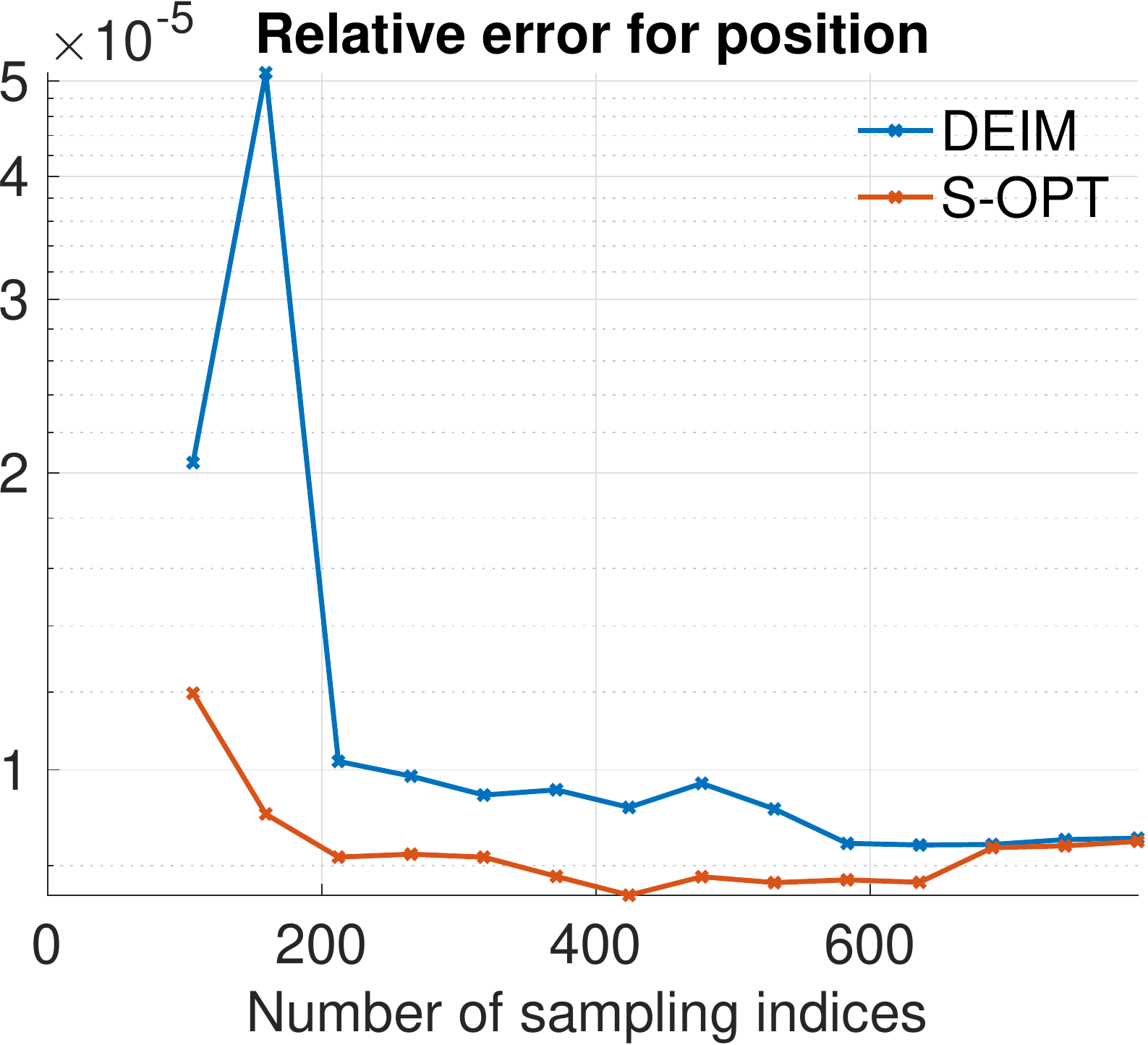}
    \hfill
    \includegraphics[width=0.32\linewidth]{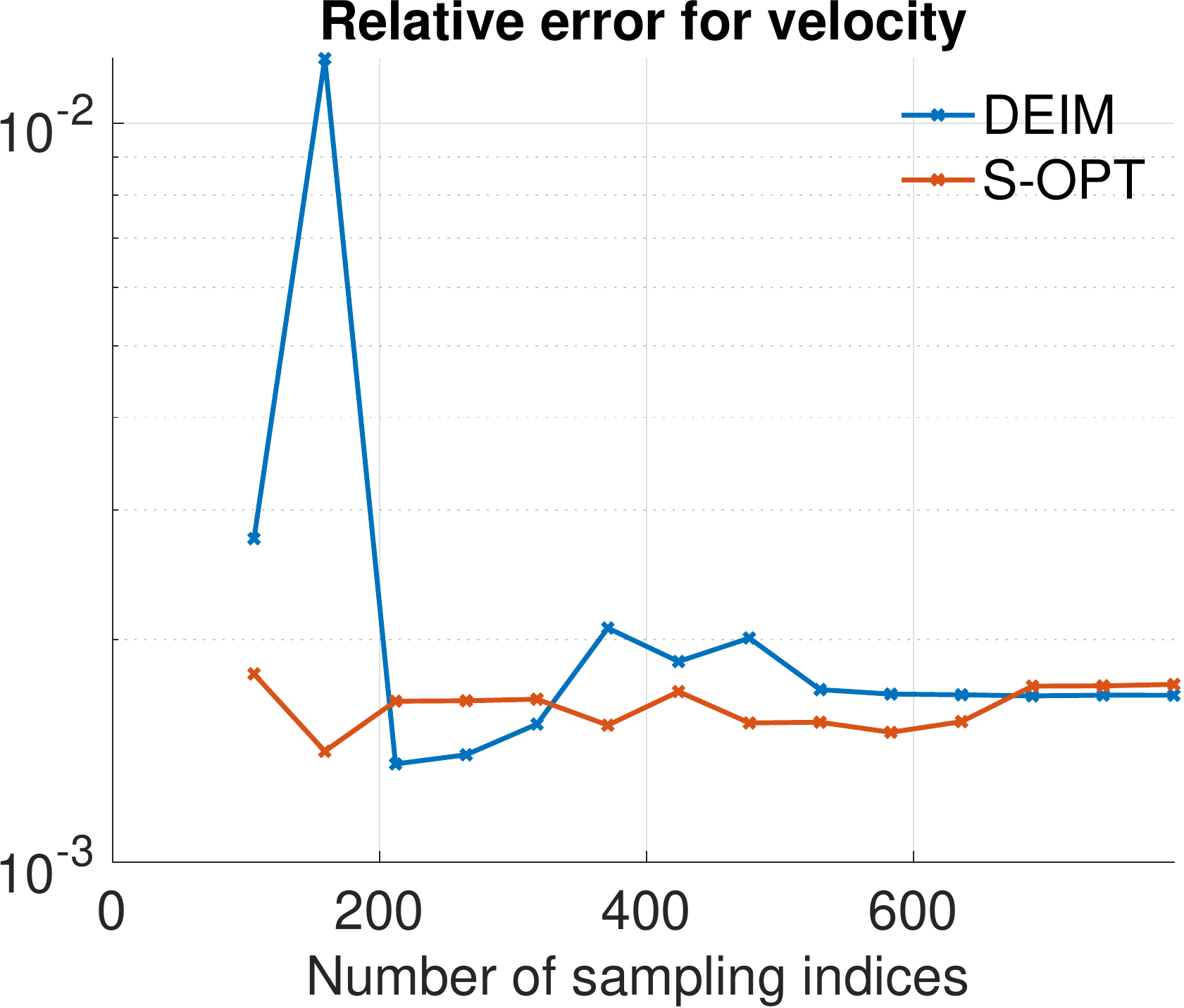}
    \hfill
    \includegraphics[width=0.32\linewidth]{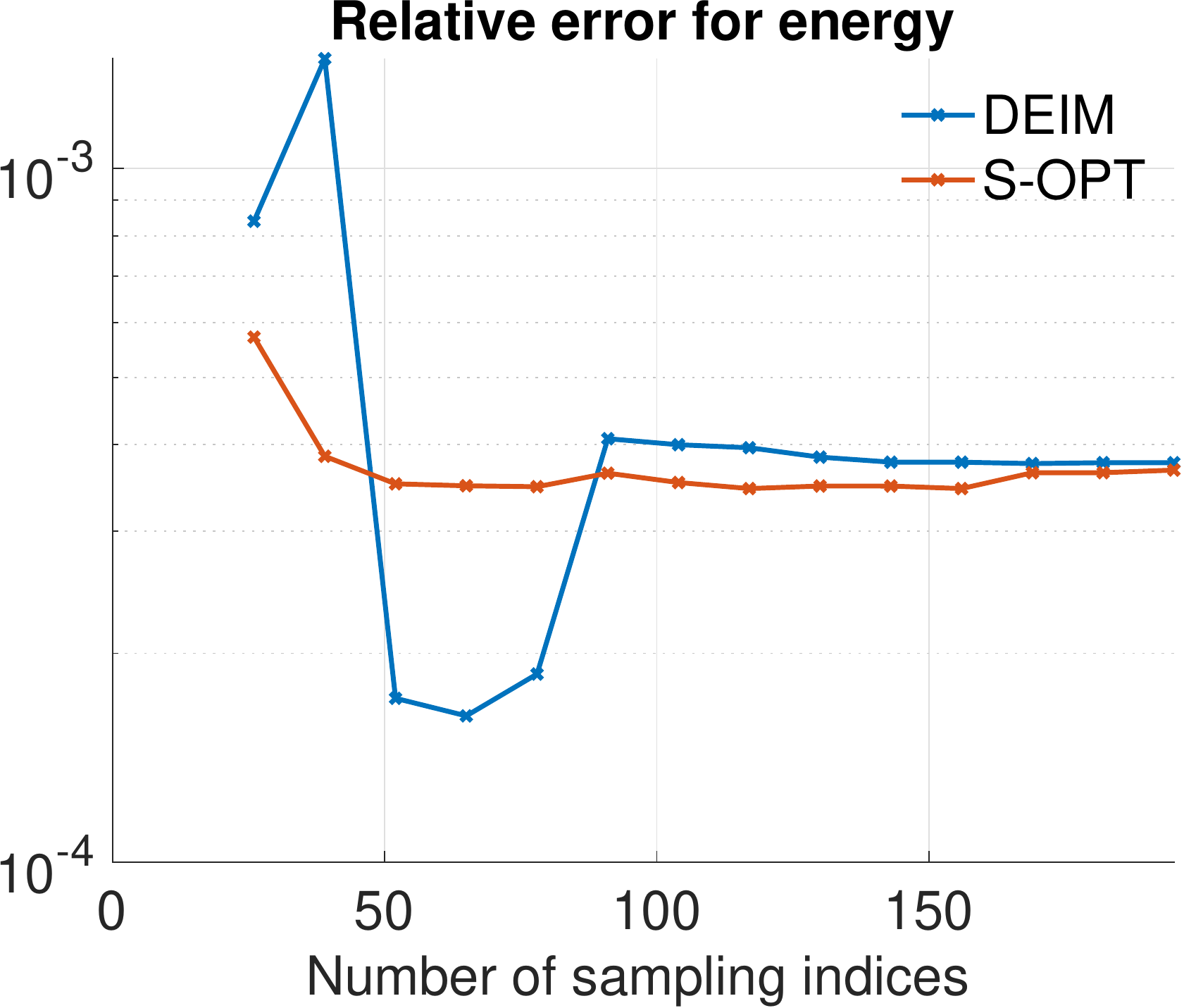}
    \caption{Final time solution error with varying number of sampling indices in Sedov blast problem.}
    \label{fig:sedov_error}
\end{figure}

\begin{figure}[htb!]
     \centering
     \begin{subfigure}[b]{0.45\textwidth}
         \centering
         \includegraphics[width=\textwidth]{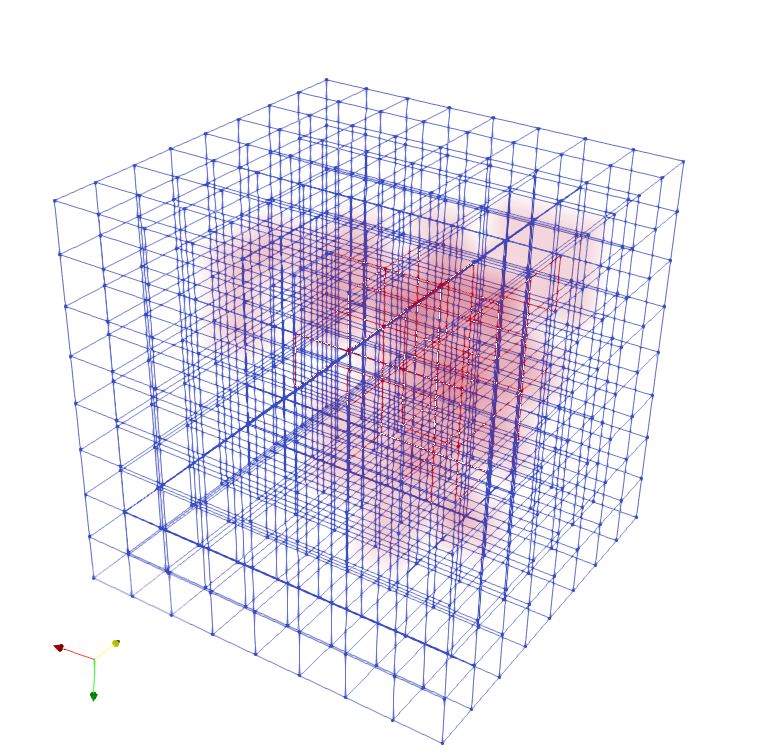}
     \end{subfigure}
     \hfill
     \begin{subfigure}[b]{0.45\textwidth}
         \centering
         \includegraphics[width=\textwidth]{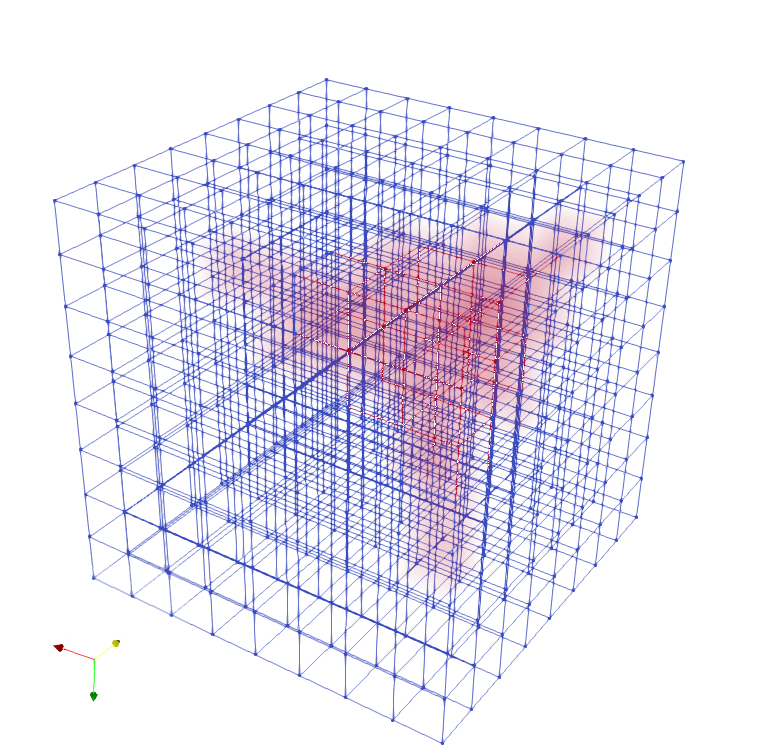}
     \end{subfigure}
        \caption{Sample mesh for Sedov blast problem of S-OPT (left) and DEIM (right) sampling algorithms.}
        \label{fig:sedov_spmesh}
\end{figure}

\section{Conclusions}
\label{sec:conclusions}

This work proposes the use of the S-OPT algorithm for selecting indices for hyper-reduction
of projection-based reduced order models. The algorithm chooses indices while trying to keep the POD 
basis orthogonal to enhance the numerical stability, while other selection methods, such as DEIM, do not. 

As shown in the results, the indices chosen by the two algorithms tend to lie in the regions 
experiencing the most change from snapshot to snapshot. However, for the S-OPT algorithm those
indices also tend to be more spread out. 

The resulting error for the ROM is smaller when using the S-OPT algorithm, especially when using
a small number of total selected indices. It is expected that the relative error of the ROM will increase
when selecting fewer indices, but the error when using the S-OPT algorithm appears to increase more smoothly,
whereas selecting fewer indices using the oversampled DEIM algorithm has a greater adverse effect on the 
ROM performance.

A topic of future research is comparing S-OPT to other selection algorithms, and investigating whether or not
there is an indication of how many indices to select for a certain error bound. 

\appendix

\section*{Acknowledgments}
This work was performed at Lawrence Livermore National Laboratory and partially funded by two LDRDs (21-FS-042 and 21-SI-006).  Lawrence Livermore National Laboratory is operated by Lawrence Livermore National Security, LLC, for the U.S. Department of Energy, National Nuclear Security Administration under Contract DE-AC52-07NA27344 and LLNL-JRNL-832493.

\section*{Disclaimer}
This document was prepared as an account of work sponsored by an agency of the
United States government.  Neither the United States government nor Lawrence
Livermore National Security, LLC, nor any of their employees makes any warranty,
expressed or implied, or assumes any legal liability or responsibility for the
accuracy, completeness, or usefulness of any information, apparatus, product, or
process disclosed, or represents that its use would not infringe privately owned
rights.  Reference herein to any specific commercial product, process, or
service by trade name, trademark, manufacturer, or otherwise does not
necessarily constitute or imply its endorsement, recommendation, or favoring by
the United States government or Lawrence Livermore National Security, LLC.  The
views and opinions of authors expressed herein do not necessarily state or
reflect those of the United States government or Lawrence Livermore National
Security, LLC, and shall not be used for advertising or product endorsement
purposes.

\bibliographystyle{siamplain}
\bibliography{references}
\end{document}